\documentclass[11pt]{amsart}
\usepackage{amscd}
\usepackage{amsfonts}
\usepackage{amsmath,amsthm,hyperref}
\usepackage[all]{xy}
\usepackage{amsmath,amsthm,hyperref}
\usepackage{amsmath,amssymb,amsthm,latexsym}
\usepackage{xcolor}
\usepackage{amsmath}
\usepackage{bm}
\usepackage{dsfont}
\usepackage{tikz-cd}

\usepackage{enumerate}
\numberwithin{equation}{section}
\newtheorem{theorem}{Theorem}[section]
\newtheorem{proposition}[theorem]{Proposition}
\newtheorem{definition}[theorem]{Definition}
\newtheorem{corollary}[theorem]{Corollary}
\newtheorem{lemma}[theorem]{Lemma}

\theoremstyle{remark}
\newtheorem{remark}[theorem]{Remark}
\newtheorem{example}[theorem]{\bf Example}
\newcommand{\R}{\mathbb{R}}
\newcommand{\C}{\mathbb{C}}
\newcommand{\D}{\mathbb{D}}

\newcommand{\T}{\mathbb{T}}
\newcommand{\dd}{\mathrm{d}}
\newcommand{\zl}{(z,\bar{z},\lambda)}

\newcommand{\violet}{\textcolor{violet}}

\newcommand{\Str}{\mathbb{S}}

\begin{document}
\title[Harmonic maps of finite uniton number via normalized extended frames]{\bf{Harmonic maps of finite uniton number into inner symmetric spaces via  based normalized extended frames}}
\author{Josef F. Dorfmeister, Peng Wang}
\maketitle

\begin{abstract}
In this paper, we develop a loop group description of harmonic maps $\mathcal{F}: M \rightarrow G/K$ of
finite uniton number,  from a Riemann surface $M,$ compact or non-compact,  into  inner symmetric spaces of
 compact or non-compact type.
As a  main result we show that  the theory of \cite{BuGu}, largely based on Bruhat cells, can be transformed into the DPW theory which is mainly based on Birkhoff cells. Moreover, it turns out that the potentials constructed in  \cite{BuGu}, mainly see section 5,  can be used to carry out the DPW procedure which uses essentially the  fixed initial condition $e$ at a fixed base point $z_0$. This extends work of Uhlenbeck, Segal, and Burstall-Guest to non-compact inner symmetric spaces as target spaces (as a consequence of a ``Duality Theorem"). It  also permits to say that there is a 1-1-relation between finite uniton number harmonic maps and normalized potentials of a very specific and very controllable type.
In particular, we prove that every harmonic map of finite uniton type from any  (compact or non-compact) Riemann surface into any  {compact or non-compact} inner symmetric space has a normalized potential taking values in some nilpotent Lie subalgebra, as well as a normalized frame with initial condition identity. This provides a straightforward way to construct all such harmonic maps. We also illustrate the above results exclusively by Willmore surfaces, since this  problem is motivated by the study of Willmore two--spheres in spheres.\vspace{5mm}
\end{abstract}

{\bf Keywords:}  harmonic maps of finite uniton type; non-compact inner symmetric spaces; normalized potential; Willmore surfaces.\\

MSC(2010): 58E20; 53C43;  53A30; 53C35

\section{Introduction}

 Harmonic maps from Riemann surfaces into symmetric spaces arise naturally in geometry and mathematical physics and hence became important objects in several mathematical fields, including  the study of  minimal surfaces, CMC surfaces, Willmore surfaces and related integrable systems. For harmonic maps into compact symmetric spaces or compact Lie groups, one of the most foundational and important papers is the  {description} of all harmonic two spheres into $U(n)$ by Uhlenbeck \cite{Uh}. It was shown that harmonic two-spheres satisfy a very  restrictive condition. Uhlenbeck coined the expression  `` finite uniton number'' for this property. Since the uniton number is an integer, one also obtains this way a subdivision of harmonic maps into  $U(n)$.  Uhlenbeck's work was generalized in a very elegant way to harmonic two spheres in all compact semi-simple Lie groups  {and into all compact inner symmetric spaces} by Burstall and Guest in \cite{BuGu}. Using  Morse theory for loop groups in the spirit of Segal's work, they showed that  {harmonic maps of finite uniton number} in compact Lie groups can be related to  {meromorphic } maps into  (finite dimensional)  nilpotent  Lie algebras. They also provided a concrete method to find all such nilpotent Lie algebras. Finally, via the Cartan embedding, harmonic maps into compact inner symmetric spaces are considered as harmonic maps into Lie groups satisfying some algebraic ``twisting'' conditions. Therefore the theory of Burstall and Guest provides a description of  {harmonic maps of finite uniton number} into compact Lie groups and compact inner symmetric spaces, in a way which is not only theoretically satisfying, but can also be implemented well for concrete computations \cite{Gu2002}.

 While in the literature  primarily harmonic maps $\mathcal F: M\rightarrow \Str$ were considered, where $\Str$ is a compact symmetric space,  in the theory of Willmore surfaces in $S^n$, and many other surface classes, one has to deal with ``Gauss type maps" which are harmonic maps into non-compact symmetric spaces.
It is the general goal of this paper to generalize results of \cite{BuGu} to harmonic maps of finite uniton type into a non-compact inner symmetric space. In particular, we want to describe simple potentials (in the sense of  {the  DPW method}) which generate such surfaces. For a compact inner symmetric target space this task has been carried out fairly explicitly in \cite{BuGu}, see subsections \ref{f.u.alaBuGu} and \ref{BuGu<->DPW}  below for a description in our notation.
For the case of a non-compact symmetric target space no such description is known yet. Thus one looks for an approach which permits to apply the work of \cite{BuGu} in such a way that one can also find simple potentials for the case of a  non-compact inner symmetric target space of a harmonic map.
  In \cite{DoWa13}, when dealing with harmonic maps into $SO^+(1,n+3)/SO^+(1,3)\times SO(n)$, the authors found a simple way to relate harmonic maps into a  non-compact inner symmetric space $G/K$  to harmonic maps into the  {dual compact} inner symmetric space $U/(U\cap K^{\C})$ dual to $G/K$.
  These two harmonic maps have a simple, but very important  relationship:  they share the  {same} meromorphic extended  framing and the normalized potential (see Theorem 1.1 of \cite{DoWa13} and Theorem \ref{thm-noncompact} in this paper). Here the normalized  extended framing and the normalized potential are meromorphic  data related to a harmonic map in terms of the language of the DPW method \cite{DPW}, which is a generalized Weierstrass type representation for harmonic maps into symmetric spaces.

  This paper is a fairly direct continuation of the paper \cite{DoWa-AT}.
    We will therefore use the same notation and definitions and assumptions.
  While we collect basic notation and results in the appendix, we would like to suggest to the reader to take a look
  at \cite{DoWa-AT} for an enhanced description of the background of this paper.

   Interpreting the work of Burstall and Guest, one will see that for harmonic maps of finite uniton type into a compact symmetric space, their work considers normalized potentials which take values  in some (fixed) nilpotent Lie subalgebra (of the originally given finite dimensional complex Lie algebra) and their extended meromorphic frames take values in the  loop group of the corresponding unipotent Lie subgroup.   Therefore, for each fixed value of the loop parameter these meromorphic extended frames
take values in  the  finite dimensional unipotent Lie group mentioned   above
(see  Theorem \ref{thm-finite-uniton2}  {in this paper}, or  Theorem 1.11 \cite{Gu2002}).
 It thus turns out that by combining the dualization procedure of  \cite{DoWa13} with the  {grouping by the uniton number} of finite uniton number harmonic maps of \cite{BuGu}, one is able to characterize all harmonic maps of finite uniton type  into non-compact inner symmetric spaces by characterizing all the normalized extended  frames and the normalized potentials of harmonic maps of finite uniton type into compact inner symmetric spaces, which, according to the theory of Burstall and Guest, can be
  {described precisely.} Simply speaking, the case of a harmonic map into a non-compact inner symmetric space $G/K$ comes exactly from the case of a harmonic map into the compact dual inner symmetric space $U/(U\cap K^{\C})$ by choosing the same normalized potential for both harmonic maps, but using the two different real forms $G$ and $U$ of $G^\C$ for the loop group construction of the corresponding harmonic maps  {(see Theorem \ref{thm-finite-uniton-n-com})}.

From a technical point of view  it is important to observe (as pointed out above already) that in \cite{BuGu} the construction of harmonic maps uses  ``extended solutions", while the loop group method  uses extended frames. It is therefore a priori difficult to relate these two construction schemes to each other. For the convenience of the reader and to fix notation we start a comparison of these two methods by recalling the relationship
between the extended solutions and the extended frames associated to a  harmonic map into a symmetric space. Then we introduce  the main results of Burstall and Guest on harmonic maps of finite uniton type \cite{BuGu}, as well as a description of their work in terms of {\em normalized potentials}, some of which has appeared in \cite{BuGu} and \cite{Gu2002}. Applying the duality theorem \cite{DoWa13}, we obtain the Burstall-Guest theory for the cases of non-compact inner symmetric spaces.
Both theories occurring in this paper consider group valued (actually ``matrix valued") function systems satisfying certain (partial) differential equations in the space variable with dependence on some ``loop parameter".
The ``extended solution approach" only fixes the solutions for two values of the loop parameter, while the ``extended frame approach = DPW method" fixes values of these frames for each loop parameter at a fixed basepoint in the domain of definition. As a simple consequence, the DPW method works with unique solutions and essentially bijective relations between potentials and harmonic maps. For simplicity we frequently say : the frame $F$ has initial condition $F(z_0, \bar z_0,\lambda)$, when we should spell out more explicitly : the initial condition of the frame $F(z, \bar z,\lambda)$ at the base point $z_0$ is $F(z_0, \bar z_0,\lambda)$.

 Finally, here comes the problem of initial conditions in the study of harmonic maps of  finite uniton type:  This turns out be crucial in the cases of non-compact inner symmetric spaces, although it is not a big problem in the compact cases. It comes from the fact that the Iwasawa decomposition for compact loop groups  is global while the Iwasawa decomposition for non-compact ones  is local (See for example
 \cite{B-R-S,oWa12,DoWa13}). Also note that the freedom of initial conditions  also is equivalent to the freedom of special dressing actions. In this sense, a fixed initial condition will simplify the classification of harmonic maps of finite uniton type further, which makes it more simple to derive geometric properties of harmonic maps via normalized potentials. A standard example is the description of minimal surfaces in $\R^n$ via potentials in \cite{Wang-2}.  In Theorem \ref{thm-finite-uniton-in} we can show that the initial condition of the meromorphic extended  frame can be set without loss of generality to be identity.

Our main motivation for the study of such harmonic maps is to provide the background for a detailed study of a wide variety of different types  of Willmore surfaces in $S^n$, surfaces of compact or non-compact type. As an application of this paper a rough classification of Willmore two-spheres  (whose conformal Gauss maps  take values in  the non-compact symmetric space $SO^+(1,n+3)/SO(1,3)\times SO(n)$) has been worked out in \cite{Wang-1}. For the convenience of the reader we  include the main result of \cite{Wang-1} by presenting its coarse classification of Willmore two-spheres in $S^{2n}$, in terms of the normalized potentials of their conformal Gauss maps. Moreover, we also present a new Willmore two-sphere constructed by using \cite{DoWa11,DoWa12} and the results of this paper. This example also responds to an open problem posed by  Ejiri in 1988 \cite{Ejiri1988}. \vspace{2mm }

This paper is organized as follows. In Section 2 we first introduce the main results of Burstall and Guest on harmonic maps of finite uniton type \cite{BuGu}, as well as a description of their work in terms of {\em normalized potentials}, most of which have appeared in \cite{BuGu} and \cite{Gu2002}. Then we prove that for harmonic maps of finite uniton numbers into compact inner symmetric spaces, the initial condition of the extended frame and the extended meromorphic frame can be identity at some chosen base point in $M$.
 Finally in Section 5, we first apply the above result to get the same description of harmonic maps of finite uniton type into non-compact inner symmetric spaces. As an illustration, an outline of applications to the study of Willmore surfaces is listed. We end the paper with three appendixes for readers' convenience:  Section 4 for  the basic results of the loop group theory for harmonic maps; Section 5, for a quick description of harmonic maps of finite uniton type; Section 6 for the duality theorem of \cite{DoWa13}.\\

{\bf Notation:} Following  \cite{BuGu,Gu2002}, in this paper,  $G$ is assumed to be a connected, semi-simple real Lie group with trivial center, and  $\mathfrak{g}$  denotes its Lie algebra and $G^{\mathbb{C}}$ its complexification. And $G/K$ denotes an inner symmetric space. We will also assume w.l.g. that $G^\C$ is a semisimple simply-connected matrix Lie group and $G$ a subgroup of $G^\C$ \cite{Hoch}, with $e$ as the identity.



\section{ Revisiting of {The work of Burstall and Guest} in the  DPW Formalism}

In this section, we will revisit the important work of
  Burstall and Guest  \cite{BuGu} on harmonic maps of finite uniton number into inner symmetric spaces, in view of the DPW method. The prime aim is to take a further consideration of the initial conditions which is not discussed in \cite{BuGu}. We show that it suffices to get all harmonic maps of finite uniton numbers by setting initial condition $e$ in the DPW version of Burstall-Guest's work. This plays an importance role when applying Burstall-Guest's work for harmonic maps of finite uniton numbers  into non-compact inner symmetric spaces because the initial condition $e$ in the non-compact case will make the characterization of harmonic maps of finite uniton numbers  into non-compact inner symmetric spaces much more simple.

\subsection{Review of extended solutions}
 In this subsection, we  compare/unify the notation used in \cite{Uh}, \cite{BuGu} and \cite{DPW}. For a harmonic map  $\mathbb{F}:\D\rightarrow G$, in \cite{Uh},  \cite{BuGu} ``extended solutions" are considered, while in \cite{DPW} always ``extended frames'' are used.  In this subsection we will explain the relation between these methods, which already appears in \cite{DoWa-AT}. Here we include primarily  the details which we will need to use.
\subsubsection{Inner symmetric spaces and the modified Cartan embedding}

Consider the inner compact symmetric space  $G/\hat{K}$  with inner involution $\sigma$, given by $\sigma(g) = h g h^{-1}$ and with $\hat{K} = Fix^\sigma (G)$. Note that $h^2 \in Center (G) = \{e\}. $  Then with $R_h(g)= gh$ we consider
the map
\begin{equation}
    \label{eq-Cartan}
\begin{tikzcd}[column sep=6mm,row sep=4mm]
G/\hat{K}  \ar{r}{\mathfrak{C}}    &   G     \ar{r}{R_h} & G \\
g   \ar{r}{}  &   g \sigma(g)^{-1} = ghg^{-1} h^{-1}  \ar{r}{} & \mathfrak{C}_h:=g h g^{-1}.
\end{tikzcd}
\end{equation}
 In this way $G/\hat K$ is an isometric, totally geodesic submanifold   of $G$ \cite{BuGu}, and ${\mathfrak{C}_h}$ will be called the ``modified Cartan embedding". Note that for outer symmetric spaces the above Cartan embedding does not apply directly.

\subsubsection{Extended frames for harmonic maps  $\mathcal{F} : M \rightarrow G/\hat{K}$ and modified  harmonic maps
$\mathfrak{C}_h\circ \mathcal{F}$}
Using the notation introduced above, consider a harmonic map  $\mathcal{F}: M \rightarrow G/\hat{K}$.

By $F : \tilde{M} \rightarrow  \Lambda G_{\sigma}$ we  denote the extended frame of $\mathcal{F}$ which is normalized to $F(z_0, \bar{z}_0, \lambda) = e$ at some base point $z = z_0$ for all $\lambda \in S^1$.
The extended frame of a harmonic map $\mathcal{F}$ actually is for each fixed $\lambda$ the frame of the corresponding immersion $\mathcal{F}_\lambda$ of the
associated family of $\mathcal{F}$. Obviously, the twisting condition in our case means
\[{\sigma(\gamma)(\lambda) = h \gamma(-\lambda)h^{-1} = \gamma(\lambda)}\hbox{
for all
$\gamma \in \Lambda G_{\sigma}$.}\]

Next we consider the composition of the family of harmonic maps  $\mathcal{F}_\lambda$ with the
modified Cartan embedding $\mathfrak{C}_h$.  In our setting, since
$\mathcal{F}_\lambda =F( z, \bar{z}, \lambda)\mod \hat{K}$,   this yields the
$\lambda-$dependent harmonic map $\mathfrak{F}^h_\lambda$ given by
 \begin{equation} \label{harmonicrelation}
  \mathfrak{F}^h_{\lambda}=
 F( z, \bar{z}, \lambda) h F( z, \bar{z}, \lambda))^{-1}.
 \end{equation}
Note that $\mathfrak{F}^h_\lambda $  is a $\lambda-$dependent harmonic map satisfying
$(\mathfrak{F}^h_\lambda )^2=e, $ where the square denotes the product in the group $G$.
Moreover, we also  have     $\mathfrak{F}^h_\lambda (z_0, \bar{z}_0,\lambda)  = h.$ Harmonic maps into $G$ satisfying these two properties will be called `modified harmonic maps".

\begin{theorem}
We retain the notation and the assumptions made just above. In particular, $z_0$ is a fixed basepoint in the Riemann surface $M$ Then
there is a bijection between harmonic maps  $\mathcal{F} : M \rightarrow G/\hat{K}$  satisfying
$\mathcal{F}(z_0,\bar{z}_0,\lambda) = e\hat{K}$ and modified harmonic maps, i.e. harmonic maps
  $\mathfrak{F}^h_{\lambda}: M \rightarrow G_h = \{ ghg^{-1}; g \in G \}$ satisfying
   $(\mathfrak{F}^h_{\lambda})^2 = e$ and  $\mathfrak{F}^h_{\lambda}(z_0, \bar{z}_0, \lambda) = h$ for all   $\lambda \in \C^*$.
  This relation is given by composition with the modified  Cartan embedding (and its inverse respectively).
  \end{theorem}

\begin{proof}
We have shown ``$\Longrightarrow$" above. Assume now we have  a harmonic map
$\mathbb{F}: M \rightarrow  {{G}_h}$ satisfying
 $\mathbb{F}^2 = e$ and  $\mathbb{F}(z_0, \bar{z}_0, \lambda) = h$ for all  $\lambda \in \C^*$.
   Since   $\mathfrak{C}_h$ is an isometric  diffeomorphism   onto its image, we consider
    (for all $\lambda \in \C^*$) the $\lambda$-dependent harmonic map
    $\mathcal{F}_\lambda = (\mathfrak{C}_h)^{-1} \circ \mathbb{F}_\lambda : M \rightarrow G/\hat{K}$.
    Clearly, then we have  $\mathcal{F}(z_0, \bar z_0, \lambda) = e\hat{K}$.  Moreover, by what was shown above, we now infer
    $\mathbb{F}(z, \bar z, \lambda) = \mathfrak{F}^h_{\lambda} (z, \bar z, \lambda) =
    F( z, \bar{z}, \lambda) h F( z, \bar{z}, \lambda))^{-1}$, where $F$ denotes the extended frame of $\mathcal{F}$.
    Since     $\mathcal{F} (z_0, \bar{z}_0, \lambda) = e$,
    $h =   F( z_0 \bar{z}_0, \lambda) h F( z_0, \bar{z}_0, \lambda))^{-1}= \mathfrak{F}^h_{\lambda} ( z_0, \bar{z}_0, \lambda)$ holds.
    \end{proof}


\subsubsection{Extended solutions for harmonic maps into Lie groups}

We  have shown, among other things, in the  subsubsections above  that it is essentially sufficient
for our purposes to  consider  harmonic maps into Lie groups $G$.
In this subsubsection we consider harmonic maps into Lie groups following the approach of
\cite{Uh} and \cite{BuGu}.
 We start by relating the different loop parameters
used in \cite{Uh} , \cite{BuGu} and \cite{DPW} respectively to each other.

 To begin with, we  recall the definition of {\em extended solutions}
following Uhlenbeck \cite{Uh},\cite{BuGu}. Let  $\D \subset \C$ be a
simply-connected domain and $\mathbb{F}: \D\ \rightarrow G$  a harmonic map.
Set
\[\mathbb{A}=\frac{1}{2} \mathbb{F}^{-1}\mathrm{d} \mathbb{F} =\mathbb{A}^{(1,0)}+\mathbb{A}^{(0,1)}.\]
Consider {for $\tilde{\lambda} \in \C^*$} the equations
\begin{equation}\label{eq-Uh1}
\left\{
\begin{split}
\partial_z \Phi \mathrm{d}z&=(1-\tilde{\lambda}^{-1})\Phi\mathbb{A}^{(1,0)},\\
\partial_{\bar{z}} \Phi \mathrm{d} \bar{z}&=(1-\tilde{\lambda})\Phi\mathbb{A}^{(0,1)}.\\
\end{split}
\right.
\end{equation}
with $\Phi: \D \rightarrow \Omega G$, where the corresponding loop parameter
 is denoted here by $\tilde{\lambda}.$
 Then, by Theorem 2.2 of \cite{Uh} (Theorem 1.1 of \cite{BuGu}),
there exists a solution $\Phi(z,\bar{z},\tilde{\lambda})$ to the above
equations such that
\begin{equation}
\Phi(z,\bar{z},\tilde{\lambda}=1)=e,~\ \mbox{and}
~ \Phi(z,\bar{z},\tilde{\lambda}=-1)= \mathbb{F} (z, \bar z)
\end{equation}
hold. This solution is unique up to multiplication by some
$\gamma \in \Omega G = \{ g \in \Lambda G^{\C}_\sigma |,  g(\lambda = 1) = I \}$ satisfying $\gamma (-1) = e$. Such solutions $\Phi$ are said to be {\bf extended solutions}.

   If we also have $\mathbb{F}(z_0)=e$ , we can also choose $\Phi(z_0,\bar{z}_0, \tilde\lambda)=e$.
Although the assumption  $\mathbb{F}(z_0)=e$ was
used in \cite{Uh}, we will, as in \cite{BuGu}, not assume  this, since it is not satisfied
 in a large part of this section. The following statement  is straightforward.

\begin{lemma}\label{lemma-es} \cite{BuGu}  Let $\Phi(z,\bar{z}, \tilde{\lambda})$ be
 {\em an
extended solution} of the harmonic map  $\mathbb{F}:\D\rightarrow G$. Let $\gamma\in \Omega G$.
Then $\gamma(\lambda)\Phi(z,\bar{z},\tilde{\lambda})$ is an extended solution of the
harmonic map $\gamma(-1)\mathbb{F}(z,\bar{z})$.
\end{lemma}

\begin{remark}
	Next we show how the ``DPW approach" without the basepoint assumption  \cite{DPW} naturally leads to
	Uhlenbeck's extended solutions \cite{Uh}.
	We follow Section 9 of \cite{Do-Es} and consider the Lie group $G$  as the outer symmetric space
	$G = (G \times G)/ \Delta,$
where the defining symmetry $\tilde{\sigma}$ is given by
	$$\tilde{\sigma} (a,b) = ( b,a)$$ and we have  $$\Delta  = \{ (a,b) \in  G \times G; a=b \}.$$
	
	For the purposes of our loop group method it is necessary to consider the $G \times G-$loop group
	$\Lambda(G \times G)$
	twisted by $\tilde{\sigma}$. We thus consider the automorphism of the loop group
	$\Lambda(G \times G) = \Lambda G \times \Lambda G$ given by
	$$\hat{\tilde{\sigma}} ((a,b)) (\lambda) = \tilde{\sigma}( a(\lambda), b(\lambda)) =
	(b(\lambda), a(\lambda)).$$

It is straightforward to verify that the twisted loop group $\Lambda(G \times G)_{\tilde{\sigma}}$
	is given by
	$$\Lambda(G \times G)_{\tilde{\sigma}} = \{ (g(-\lambda), g(\lambda)) ; g(\lambda) \in \Lambda G \}
	 \cong \Lambda G. $$

	Let's consider now a harmonic map $\mathbb{F}: M \rightarrow G.$
	Then the map
	$\mathfrak{F}:M \rightarrow G \times G$, given by
	 $\mathfrak{F}(z, \bar{z}) = \left(\mathbb F(z,\bar{z}),e\right)$,  is a global frame of $\mathbb{F}$.
	
	Following \cite{DPW} one needs to decompose the Maurer-Cartan form
	$\mathfrak{A} = \mathfrak{F}^{-1} \mathrm{d} \mathfrak{F}$ of $\mathfrak{F}$  into the eigenspaces of
	$\tilde\sigma$ and to introduce the loop parameter $\lambda$.
	One obtains (see \cite{Do-Es}, formula $(68)$):
	\begin{equation}
	\mathfrak{A}_{\lambda} = \left( \left(1+\lambda^{-1}\right) \mathbb A^{(1,0)} +  \left(1+\lambda\right)\mathbb A^{(0,1)} ,\ \left(1-\lambda^{-1}\right) \mathbb A^{(1,0)}  +  \left(1-\lambda\right) \mathbb A^{(0,1)} \right).
	\end{equation}
\end{remark}
	\begin{theorem}
		Let  $G$  be a connected, compact or non-compact, semi-simple real Lie group with trivial center.
		Let  $\mathbb{F}:\D\rightarrow G$ be a harmonic map.  Then, when  representing $G$ as the
		symmetric space  $G = (G \times G)/\Delta,$
		any extended  frame $\mathfrak{F} : \D \rightarrow \Lambda(G \times G)_{\tilde\sigma}$
		of $\mathbb{F}$ satisfying
		$\mathfrak{F}(z, \bar{z}, \lambda = 1) = \left(\mathbb F(z,\bar{z}),e\right)$
		is given  by a pair of functions,
		\[\mathfrak{F}(z, \bar{z},\lambda) =  ( \Phi (z, \bar{z},-\lambda), \Phi (z, \bar{z},\lambda)),\]
		where the  matrix function $ \Phi (z, \bar{z},\lambda)$ is an  extended solution  for $\mathbb{F}$ in the sense of Uhlenbeck \cite{Uh} as introduced above.
	\end{theorem}
	
	\begin{proof}
		Since the two components of $\mathfrak{A} $ only differ by a minus sign in $\lambda$, any solution to
		the equation $\mathfrak{A}_\lambda  = \mathfrak{F}(z, \bar z, \lambda)^{-1} \mathrm{d} \mathfrak{F}(z, \bar z, \lambda)$
		is of the form $ \mathfrak{F}(z, \bar z, \lambda) =
		(B(-\lambda) \Psi(z, \bar z, - \lambda) , B(\lambda) \Psi(z, \bar z,  \lambda))$,
		where $\Psi$ solves the equations \eqref{eq-Uh1}. Moreover, we can assume w.l.g. that $\Psi$ satisfies the two conditions for extended solutions stated above for $\lambda = \pm1$.
		Now
		$\mathfrak{F}(z, \bar{z}, \lambda = 1) = \left(\mathbb F(z,\bar{z}),e\right)$ implies
		{$B(1)  = B(1) \Psi (z, \bar{z},\lambda= 1) =e$ and $B( -1)\Psi(z, \bar z, -1) =
		 B(-1) \mathbb{F}(z, \bar z) = \mathbb{F}(z, \bar z),$
		 whence $B(-1) = e$ follows.
		 Therefore, $\Phi (z, \bar z, \lambda)  = B(\lambda)  \Psi(z, \bar z, \lambda)$}
		yields the claim.
	\end{proof}

	Note that in this theorem no normalization is required. Moreover, the loop parameter used in \cite{Uh} is the same as the one used in \cite{DPW}. However, the matrix functions $B(\lambda)$ and  $\Phi (z, \bar z, \lambda)$
	are not uniquely determined which causes the DPW procedure to yield quite arbitrary potentials, not easily permitting any converse construction procedure.

\subsubsection{Extended solutions and extended frames for harmonic maps into symmetric spaces} \label{414}

Consider as before a harmonic map  $\mathcal{F}: M \rightarrow G/\hat{K}$ into a symmetric space
with inner involution $\sigma$, given by $\sigma(g) = h g h^{-1}$ and with $\hat{K} = Fix^\sigma (G)$.
 As above we consider the modified harmonic map $\mathbb{F}: \rightarrow G$ given by
 $$\mathbb{F}(z, \bar z, \lambda) =   \mathfrak{F}^h_{\lambda} (z, \bar z, \lambda) =
 F( z, \bar{z}, \lambda) h F( z, \bar{z}, \lambda))^{-1}.$$
  For this $\lambda-$family of harmonic maps $\mathbb{F}_\lambda$ we compute
 \begin{equation*}
 \begin{split}
\mathbb{A}&=\frac{1}{2}\mathbb{F}_{\lambda}^{-1}\mathrm{d}\mathbb{F}_{\lambda} \\
\ &=\frac{1}{2}\left(F(z,\bar{z},\lambda)hF(z,\bar{z},\lambda)^{-1}\right)^{-1}\mathrm{d}\left(F(z,\bar{z},\lambda)hF(z,\bar{z},\lambda)^{-1}\right)\\
\ &=\frac{1}{2}F(z,\bar{z},\lambda) h^{-1}\alpha_{\lambda} h F(z,\bar{z},\lambda)^{-1}-\frac{1}{2}\mathrm{d}F(z,\bar{z},\lambda)F(z,\bar{z},\lambda)^{-1}\\
\ &=\frac{1}{2}F(z,\bar{z},\lambda)\left(\alpha_{-\lambda}-\alpha_{\lambda}\right)F(z,\bar{z},\lambda)^{-1}\\
\ &=-F(z,\bar{z},\lambda)\left(\lambda^{-1}\alpha_{\mathfrak{p}}'+\lambda\alpha_{\mathfrak{p}}''\right)F(z,\bar{z},\lambda)^{-1}.\\
\end{split}
\end{equation*}
Following Uhlenbeck's approach we need to introduce a new ``loop parameter" $\tilde{\lambda}$ now and consider Uhlenbeck's differential equation \eqref{eq-Uh1} for
$\Phi(z,\bar{z},\lambda,\tilde{\lambda})$ with conditions for $\tilde{\lambda} = \pm 1$:
\begin{equation}\label{eq-Uh2}
\left\{
\begin{split}
\partial_z \Phi \mathrm{d}z&=-\Phi(1-\tilde{\lambda}^{-1})\lambda^{-1}F(z,\bar{z},\lambda) \alpha_{\mathfrak{p}}' F(z,\bar{z},\lambda)^{-1}\\
\partial_{\bar{z}} \Phi\mathrm{d}\bar z&=-\Phi(1-\tilde{\lambda})\lambda F(z,\bar{z},\lambda) \alpha_{\mathfrak{p}}'' F(z,\bar{z},\lambda)^{-1}\\
\end{split}
\right.\end{equation}
From (\ref{eq-Uh2}) it is natural to consider the Maurer-Cartan form $\widetilde{\mathbb{A}} $ of $\Phi (z,\bar{z},\lambda,\tilde{\lambda})F(z,\bar{z},\lambda)$. One obtains:
\begin{equation}
\begin{split}
\widetilde{\mathbb{A}}&=(\Phi F)^{-1} \dd ( \Phi F)\\
&=F^{-1}\dd F+F^{-1}(\Phi^{-1}\dd  \Phi) F\\
&=\alpha_{\lambda}-F^{-1}\Phi^{-1}\left(\Phi(1-\tilde{\lambda}^{-1})\lambda^{-1}F  \alpha_{\mathfrak{p}}' F ^{-1}+\Phi(1-\tilde{\lambda})\lambda F \alpha_{\mathfrak{p}}'' F ^{-1}\right)F\\
&=\alpha_{\lambda}- \left( (1-\tilde{\lambda}^{-1})\lambda^{-1}  \alpha_{\mathfrak{p}}'  +(1-\tilde{\lambda})\lambda \alpha_{\mathfrak{p}}''\right)\\
&=\lambda^{-1}  \alpha_{\mathfrak{p}}'+\alpha_{\mathfrak k}+\lambda \alpha_{\mathfrak{p}}''- \left( (1-\tilde{\lambda}^{-1})\lambda^{-1}  \alpha_{\mathfrak{p}}'  +(1-\tilde{\lambda})\lambda \alpha_{\mathfrak{p}}''\right)\\
&=\tilde{\lambda}^{-1}\lambda^{-1}  \alpha_{\mathfrak{p}}'+\alpha_{\mathfrak k}+\tilde{\lambda}\lambda \alpha_{\mathfrak{p}}'' \\
&=\alpha_{\tilde{\lambda}\lambda}. \\
\end{split}
\end{equation}
From this we derive immediately the relation
\begin{equation}\label{eq-lawson}
F(z,\bar{z}, \lambda\tilde{\lambda})=A(\lambda, \tilde{\lambda}) \Phi (z,\bar{z},\lambda,\tilde{\lambda})F(z,\bar{z}, \lambda).
\end{equation}
Substituting here  $z = z_0$ we derive, in view of the normalization of $F$ at $z = z_0$:
\begin{equation}\label{eq-A}
A(\lambda, \tilde{\lambda}) = \Phi (z_0 ,{\bar{z}}_0,\lambda,\tilde{\lambda})^{-1}.
\end{equation}
In particular, setting $\lambda=1$ in \eqref{eq-lawson} we obtain
\begin{equation}
F(z,\bar{z}, \tilde{\lambda})=A(1, \tilde{\lambda}) \Phi (z,\bar{z},1,\tilde{\lambda})F(z,\bar{z},1).
\end{equation}
Setting $\tilde\lambda=-1$ in \eqref{eq-lawson} we obtain (by using the twisting condition for $F$):
\begin{equation}
A(\lambda, -1)=F(z,\bar{z},-\lambda)\left( \Phi (z,\bar{z},\lambda,-1)F(z,\bar{z},\lambda)\right)^{-1}=\violet{h}.
\end{equation}
Hence
\[\Phi (z, \bar{z},1,-1)={A(1,-1)^{-1} } F(z,\bar{z}, -1)F(z,\bar{z}, 1)^{-1}=F(z,\bar{z}, 1)hF(z,\bar{z}, 1)^{-1}=\mathbb{F}(z,\bar z,1).\]

In summary we obtain (by setting $\lambda = 1$ and replacing $\tilde{\lambda}$ by $\lambda$):
\begin{corollary}\label{cor-Phi-F}
The extended solution $\Phi$, and the $\sigma-$twisted  extended frame $F$ satisfy
\begin{equation}\label{eq-lawson2}
\Phi (z, \bar{z},1,\lambda) = A(1, \lambda)^{-1}F(z,\bar{z},  \lambda)F(z,\bar{z}, 1)^{-1}.
\end{equation}
In particular, $\Phi  (z, \bar{z},1,\lambda)$ is contained in the based loop group $\Omega G$. Moreover,
for $\lambda = -1$  we obtain the harmonic map $ \mathbb{F}(z,\bar z,1) =  \mathbb{F}(z,\bar z) $.
\end{corollary}

\subsection{Finite  {uniton type   \`{a}}
 la Burstall-Guest for harmonic maps into compact Lie groups}

 Let us recall that in Definition \ref{def-uni} we have given the definition of harmonic maps of finite uniton type into Lie groups. Now we want to define the notion of a ``finite uniton number''.
  {It has been introduced by Uhlenbeck \cite{Uh} for $U(n)$ and by Burstall-Guest
 \cite{BuGu} for a general compact real Lie group $G$.}

\begin{definition}  \label{def-f.u.}
Let $\mathbb{F}:M \rightarrow G$  be a harmonic map into a real Lie group $G$.
Assume there exists a global extended
solution $\Phi(z,
\bar{z},\tilde\lambda):M\rightarrow \Lambda G^{\mathbb{C}}$ (i.e., $\mathbb{F}$ has trivial monodromy).
We say that $\mathbb{F}$ has {\it finite uniton number $k$} if
(see \eqref{eq-alg-loop} for the definition of $ \Omega^k_{alg} G $)
 \begin{equation}\Phi(M)\subset \Omega^k_{alg} G ,\
\hbox{ and } \Phi(M)\nsubseteq \Omega^{k-1}_{alg} G .\end{equation}
 In this case we write  $r(\Phi)=k$ and the minimal uniton number of $\mathbb{F}$ is defined
as \[r(\mathbb{F}):=min\{r(\gamma  Ad(\Phi))| \gamma\in \Omega_{alg} Ad G \}.\]
\end{definition}

 It is important to this paper that the notion `` finite uniton number"  harmonic maps and  ``finite uniton type  harmonic maps" describe the same class of harmonic maps.
\begin{proposition} \label{typeequivnumber}\cite{DoWa-AT}
$\mathcal F$ is a harmonic map of finite uniton type in $G/K$ if and only if
$\mathbb{F} = \mathfrak{C}_h \circ \mathcal{F}$ is a harmonic  map of finite uniton number, where $\mathfrak{C}_h$ is the modified  Cartan  embedding \eqref{eq-Cartan}.
\end{proposition}

\begin{remark}\
\begin{enumerate}
\item The main goals of this paper is a characterization of the normalized potentials
of all finite uniton type harmonic maps into $G/K$.
The potential for such a  harmonic map is the same as the potential
for the induced harmonic map into $G/\hat{K}$,
where $\hat{K} = G^{\sigma}$, since the different Cartan maps have the same images $\mathfrak{C}(gK)=\mathfrak{C}(g\hat{K})$ for all $g\in G$. We will therefore always assume in this section that
$\hat{K} = G^{\sigma}$. In this case the Cartan map $\mathfrak{C}$ actually is an embedding.

\item   {If $M$ is simply connected, then a global extended solution
$\Phi(z,\bar{z},\lambda):M\rightarrow \Omega G $ always exists, including the case $M=S^2$ (see Theorem 2.2 of \cite{Uh}, also see \cite{Segal}, Theorem 1.1 of \cite{BuGu})}.
This is in contrast to the case of extended
 frames, in which case we have explained above that on $S^2$ the extended frame needs to have (two)
 singularities due to the topology of $S^2$. In general, the extended solution may not exist globally on $M$ if $M$ is not simply connected.
\end{enumerate}
\end{remark}
\vspace{2mm}
Now let us turn to the  Burstall-Guest theory for harmonic maps into Lie groups of  {finite uniton number}.
 Let $\mathrm{T}\subset G $  be a maximal  {torus of $G$}  with $\mathfrak{t}$ the Lie algebra of $\mathrm{T}$.
 We can identify all the homomorphisms from $S^1$ to $\mathrm{T}$ with the integer lattice $\mathcal{I}:=(2\pi)^{-1}\exp^{-1}(e)\cap\mathfrak{t}$  in $\mathfrak{t}$ via the map
\begin{equation}
\begin{array}{llllll}
\mathcal{I}=(2\pi)^{-1}\exp^{-1}(e)\cap\mathfrak{t}&\longrightarrow \hbox{\{homomorpisms from $S^1$ to $\mathrm{T}$\}}, \\
\ \ \ \ \ \ \ \ \ \ \ \ \xi  \ \ \ &\longmapsto \gamma_{\xi},\\
\end{array}
\end{equation}
where  $\gamma_{\xi}:S^1\rightarrow T$ is defined by
\begin{equation}\gamma_{\xi}(\lambda):=\exp(t\xi),\ \ \hbox{ for all }\ \ \lambda=e^{it}\in S^1.
\end{equation}
Let  {$\mathcal {C}_0$} be a fundamental Weyl chamber of $\mathfrak{t}$. Set $\mathcal{I}'=\mathcal {C}_0\cap \mathcal{I}$. Then $\mathcal{I}'$ parameterizes the
conjugacy classes of homomorphisms $S^1\rightarrow G$.
Let $\Delta$ be the set of roots of $\mathfrak{g}^{\mathbb{C}}$. We have the root space decomposition
$\mathfrak{g}^{\mathbb{C}}=\mathfrak{t}^{\mathbb{C}}\oplus (\underset{\theta\in\Delta}{\oplus}\mathfrak{g}_{\theta} ).$  Decompose $\Delta$ as $\Delta=\Delta^-\cup\Delta^+$ according to  {$\mathcal {C}_0$}.
 Let $\theta_1,\cdots,\theta_l\in \Delta^{+}$ be the simple roots. We denote by $\xi_1,\cdots,\xi_l\in \mathfrak{t}$ the basis of $\mathfrak{t}$ which is dual to
$\theta_1,\cdots,\theta_l$ in the sense that $\theta_j(\xi_k)=\sqrt{-1}\delta_{jk}$.

\begin{definition} $($  {p.555 of \cite{BuGu}}$)$
An element $\xi$ in $\mathcal{I}'\backslash\{0\}$ is called a {\em canonical} element, if $\xi=\xi_{j_1}+\cdots+\xi_{j_k}$ with $\xi_{j_1},\cdots,\xi_{j_k}\in \{\xi_1,\cdots,\xi_l\}$ pairwise different. In other words, for every simple root $\theta_j$, we have that $\theta_{j}(\xi)$ only attains the values $0$ or $\sqrt{-1}$.\end{definition}

For $\theta\in\Delta$ and $X\in\mathfrak{g}_{\theta}$ we obtain
$$ad\xi X=\theta(\xi)X\ \ \hbox{ and }\ \theta(\xi)\in\sqrt{-1}\mathbb{Z}.$$
Let $\mathfrak{g}^{\xi}_{j}$ be the $\sqrt{-1}\cdot j-\hbox{eigenspace}$ of $\hbox{ad}\xi$. Then
 \begin{equation} \label{defgjxi}
\mathfrak{g}^{\xi}_j=\underset{\theta(\xi)=\sqrt{-1} j}{\oplus}\mathfrak{g}_{\theta}, \hbox{ and } \mathfrak{g}^{\mathbb{C}}=\underset{j}{\oplus}\ \mathfrak{g}^{\xi}_j.
\end{equation}
We define the {\em height} of $\xi$ as the non-negative integer \begin{equation}r(\xi)=\hbox{ max}\{j|\ \mathfrak{g}^{\xi}_j\neq0\ \}.\end{equation}

\begin{lemma}\label{lemma-xi} (Lemma 3.4 of \cite{BuGu})  Let $\xi=\sum_{i=1}^kn_{j_i}\xi_{j_i}\in \mathcal{I}'$ with $n_{j_i}>0$. Set $\xi_{can}=\sum_{i=1}^k\xi_{j_i}$. Then we have
 \begin{equation}
 \mathfrak{g}^{\xi}_0=\mathfrak{g}^{\xi_{can}}_0,\ \ \ ~~ \sum_{0\leq j\leq r(\xi)-1} \mathfrak{g}^{\xi}_{j+1}=\sum_{0\leq j\leq r(\xi_{can})-1}  \mathfrak{g}^{\xi_{can}}_{j+1}.
 \end{equation}
\end{lemma}
Set
\begin{equation}\left\{
\begin{array}{llllll}
&\mathfrak{f}^{\xi}_j&:= &\underset{k\leq j}{\oplus}\mathfrak{g}^{\xi}_k,\\
 &(\mathfrak{f}^{\xi}_j)^{\perp}&:=&\underset{j< k \leq r(\xi)}{\oplus} \mathfrak{g}^{\xi}_k, \\
 &\mathfrak{u}^0_{\xi}&:=&\underset{0\leq j <r(\xi)} {\oplus}\lambda^{j}(\mathfrak{f}^{\xi}_j)^{\perp}\in \Lambda^+\mathfrak{g}^\C_{\sigma}.\\
\end{array}\right.
\end{equation}
 {Now we can state (in our notation) some of the main results of \cite{BuGu}.}

\begin{theorem}\label{thm-finite-uniton0} (Theorem 1.2, Theorem 4.5, and  {p.560} of \cite{BuGu})
Assume $G$ is connected, compact, and semisimple with trivial center.
\begin{enumerate}
  \item Let $\Phi:M\rightarrow \Omega^k_{alg}G$ be an extended solution of finite uniton number. Then there exists some canonical $\xi\in \mathcal{I}'$, some $\gamma\in \Omega_{alg}G$, and some discrete subset $D'\subset M$, such that on $M\setminus D'$,  {the following Iwasawa decomposition of $\exp C \cdot \gamma_{\xi}$ holds:}
\begin{equation}\label{eq-finite}
 {\gamma\Phi=\exp C \cdot \gamma_{\xi} \cdot (\Phi^+_{\xi})^{-1},}
\end{equation}
where $C:M\rightarrow \mathfrak{u}^0_{\xi}$ is a (vector-valued) {\em meromorphic  {function}} with poles in $D'$
and  {$\Phi^+_{\xi}: M \backslash D' \rightarrow \Lambda^+G^\C_{\sigma}$. Moreover, the Maurer-Cartan form of $\exp C$ is given  by}
\begin{equation}\label{eq-C}
 {(\exp C)^{-1}\dd(\exp C)=\sum_{0\leq j\leq r(\xi)-1}\lambda^j A_j'\dd z,}
\end{equation}
 {with $A_j':M\rightarrow \mathfrak{g}^{\xi}_{j+1}$ being meromorphic functions with poles contained in $D'$  for each $j$.}

 \item
 Conversely, let $\xi\in \mathcal{I}'$ be a canonical element and  $C:M \rightarrow \mathfrak{u}^0_{\xi}$  a meromorphic function satisfying \eqref{eq-C}. Let $D'\subset M$ be the set of poles of $C$,
 {and  $\Phi=(\exp C \cdot \gamma_{\xi}) \cdot (\Phi^+_{\xi})^{-1} $  be an Iwasawa decomposition of $ \exp C \cdot \gamma_{\xi}$. Then $\Phi$ is an extended solution of finite uniton number on $M$.}
\end{enumerate}
\end{theorem}

\begin{remark}\
  \begin{enumerate}
\item
 {Let $z_0\in M\setminus D'$ be some point. Then $C_0=C(z_0)\in \mathfrak{u}^0_{\xi}$
 and $\Phi(z_0,\bar{z}_0,\lambda)\in\Omega^k_{alg}G$ gives some initial condition.}
  \item By Lemma \ref{lemma-es}, since $\Phi(z,\bar{z},\lambda)$ is an extended solution of a harmonic map
  $\Phi(z,\bar{z},-1)$, then  $\gamma\Phi(z,\bar{z},\lambda)$ is  an extended solution of the harmonic map
   $\gamma(-1)\Phi(z,\bar{z},-1)$, which is congruent to the harmonic map $\Phi(z,\bar{z},-1)$ up to the group action of $G$ from the left.  {So for the harmonic map $\gamma(-1)\Phi(z,\bar{z},-1)$ we have an extended solution for which without loss of generality we can always assume $\gamma=e$ in $(1)$ of the theorem above.}

  \item
  On page 560 of \cite{BuGu}, the elements $A'_j$ are defined so that they take values in $\mathfrak{f}^{\xi}_{j+1}=\sum_{k\leq j+1}\mathfrak{g}^{\xi}_{k}$ which is due to the harmonicity of $\mathbb{F}$. In view of the restriction on $C$ to take values in $\mathfrak{u}_{\xi}^0$,  the computations on page 561 of \cite{BuGu} for $(\exp C)^{-1}(\exp C)_z$  imply that
$A'_j$  takes values in $\sum_{k\geq j+1}\mathfrak{g}^{\xi}_{k}$.  Therefore $A'_j \in \mathfrak{g}^{\xi}_{j+1}$ as stated above.
\end{enumerate}
\end{remark}


\subsection{Finite uniton  {type \`{a} la} Burstall-Guest for harmonic maps into symmetric spaces} \label{f.u.alaBuGu}

 {Consider a harmonic map $\mathcal{F}: M \rightarrow G/K$ from $M$ into the inner symmetric space
$G/\hat K$. As stated in \cite{BuGu}, the inner symmetric space $G/\hat{K}$,  can be embedded into $G$ via the modified Cartan map $\mathfrak{C}_h$  such that  $\mathfrak{C}_h(G/\hat{K})$ is a connected component of
$\sqrt{e}$,\  where $\sqrt{e}:=\{g\in G| g^2=e \}.$}
 Let  $$ {\T}:\Omega G\rightarrow \Omega G,\ \  {\T}(\gamma)(\lambda)=\gamma(-\lambda)\gamma(-1)^{-1}$$
be an involution of $\Omega G$ (see Section 2.2) with the fixed point set
\[(\Omega G)_{ {\T}}=\{\gamma\in\Omega G|  {\T}(\gamma)=\gamma\}.\]
 { We give a version of Proposition 5.2 of \cite{BuGu} adapted to our approach.}

\begin{proposition}\label{eq-check-xi} { Let $\mathcal F$ be a harmonic map of finite uniton number with an extended frame $F(z,\bar z, \lambda)$. Let $\Phi(z,\bar z,1,\lambda)=A(1, {\lambda})^{-1} F(z,\bar{z}, \lambda)F(z,\bar{z}, 1)^{-1}$ be an extended solution for the harmonic map $\mathbb{F}=\mathfrak{C}_h\circ\mathcal F$ as described in
(\ref{eq-lawson2}) of Corollary \ref{cor-Phi-F}.} Then there exists some $\tilde\xi \in \mathfrak{t}$  satisfying
$\exp( \pi \tilde{\xi}) = h$, where $\mathfrak{t}$  denotes the Lie algebra of  {the maximal torus
$\mathrm{T}$ in $G$.}
 Setting $\gamma_{\tilde\xi} (\lambda) = \exp (t \tilde\xi)$ for $\lambda = e^{it}$ and $\tilde{\gamma} (\lambda) = \gamma_{\tilde\xi} (\lambda) A(1, \lambda)$ we obtain
$\tilde{\gamma}(\lambda)\in\Omega G$ , $\tilde{\gamma}(-1)=e$ and
\begin{equation}\label{eq-TP}
 {\T}(\tilde{\gamma}(\lambda)\Phi(z,\bar z,1,\lambda))=\tilde{\gamma}(\lambda)\Phi(z,\bar z,1,\lambda).
 \end{equation}
 Moreover,  $\tilde{\gamma}(\lambda)\Phi$ is also an extended solution for
 $\mathbb F$.
Furthermore, if $\Phi(z,\bar z,1,\lambda)$ takes  {values} in $\Omega _{alg}G$, then we also have  $\tilde{\gamma}(\lambda)\in \Omega _{alg}G$ and  $\tilde{\gamma}(\lambda)\Phi(z,\bar z,1,\lambda)$ takes  {values} in $(\Omega _{alg}G)_{ {\T}}$.
\end{proposition}

\begin{proof} The proof follows closely the one of \cite{BuGu}. First we note that $h$ is contained in the maximal torus. This implies the existence of $\tilde\xi$.
Next we verify \eqref{eq-TP}. Since \[\gamma_{\tilde\xi}(-\lambda)=\exp((\pi+t)\tilde \xi)=
\gamma_{\tilde\xi}(\lambda)h\ \hbox{ and }\ F(z,\bar{z}, -\lambda)h=hF(z,\bar{z}, \lambda),\] we have
\[\begin{split}
 {\T}(\tilde{\gamma}(\lambda)\Phi(z,\bar z,1,\lambda))
&= {\T}\left(\gamma_{\tilde\xi}(\lambda)F(z,\bar{z}, \lambda)F(z,\bar{z}, 1)^{-1}\right)\\
&=\gamma_{\tilde\xi}(-\lambda)F(z,\bar{z}, -\lambda)F(z,\bar{z}, 1)^{-1}F(z,\bar{z}, 1)F(z,\bar{z}, -1)^{-1}\gamma_{\tilde\xi}(-1)\\
&=\gamma_{\check\xi}(\lambda)hF(z,\bar{z}, -\lambda)F(z,\bar{z}, -1)^{-1}h\\
&=\gamma_{\check\xi}(\lambda)F(z,\bar{z}, \lambda)hF(z,\bar{z}, 1)^{-1}\\
&=\tilde{\gamma}(\lambda)\Phi(z,\bar z,1,\lambda).
\end{split}\]
\end{proof}
 {By Proposition \ref{eq-check-xi} we can assume without loss of generality that $\Phi$ has the form \begin{equation}\label{eq-Phi-F}
\Phi=\gamma_{\tilde\xi}(\lambda)F(z,\bar{z}, \lambda)F(z,\bar{z}, 1)^{-1}.\end{equation}}
With this $\Phi$ we obtain:

\begin{theorem}\label{thm-finite-uniton1}(Proposition 5.3, Theorem 5.4  {and p.567 of } \cite{BuGu})
Assume that $G$ is connected, compact, and semisimple with trivial center.

  \begin{enumerate}
  \item Let $\Phi:M\rightarrow (\Omega^k_{alg}G)_\T$ be an extended solution for some harmonic map
   { $\mathcal F$  of finite uniton number related by \eqref{eq-Phi-F}.}  Then there exists some canonical element $\xi$ in $\mathcal{I}'$, some $\gamma \in \Omega_{alg}G$  {satisfying $\gamma(-1) = e,$ }and some discrete subset $D'$ of $M$ such that on $M\setminus D'$,
\begin{equation}
  { \gamma\Phi =\exp C \cdot \gamma_{\xi} \cdot (\Phi^+_{\xi})^{-1},}
 \end{equation}
where  {$C:M\rightarrow (\mathfrak{u}^0_{\xi})_\T$ is a {\em meromorphic map} on $M$ with poles in $D'$} and $$(\mathfrak{u}^0_{\xi})_\T=\bigoplus_{0\leq 2j <r(\xi)}\lambda^{2j}(\mathfrak{f}^{\xi}_{2j})^{\perp}.$$
Moreover, $\xi$  satisfies \begin{equation}G/\hat{K} \cong\{g(\exp \pi \xi)g^{-1}| g\in G\}.
\end{equation}

 {Note that both $\Phi$ and $\gamma \Phi$ are extended solutions of $\mathcal F$.}
  \item
  Conversely, let  $\xi\in \mathcal{I}'$ be a canonical element.
    {Assume that $C:M \rightarrow (\mathfrak{u}^0_{\xi})_\T$ is a meromorphic function} such that
 {\begin{equation}
(\exp C)^{-1}\dd(\exp C)=\sum_{0\leq 2j\leq r(\xi)-1}\lambda^j A_{2j}'\dd z,\
\end{equation}
with $A_{2j}':M \rightarrow \mathfrak{g}^{\xi}_{2j+1}$ being meromorprhic,
$\ 0\leq 2j\leq r(\xi)-1$.}
Let $D'\subset M$ be the set of poles of $C$, then
 {$\Phi_{M\setminus D'}=\gamma_{\xi}^{-1} \cdot  \exp C  \cdot \gamma_{\xi} \cdot (\Phi^+_{\xi})^{-1} $} is an extended solution  { for some harmonic map of finite uniton number  $F:M\setminus D' \rightarrow G/\hat{K}\cong\{g(\exp \pi \xi)g^{-1}| g\in G\}\subset G$.}
 \end{enumerate}\end{theorem}

\begin{remark}  A   generalization of Burstall and Guest's theory for outer compact symmetric spaces has been published in \cite{Esch-Ma-Qu}.
\end{remark}

\subsection{The Burstall-Guest theory in relation to standard DPW theory} \label{BuGu<->DPW}


Using the above theorems, we can derive the normalized potential of harmonic maps of finite uniton type, showing that they are meromorphic 1-forms taking values in a fixed nilpotent Lie algebra.

  {Note that these results have been explained in Appendix B of \cite{BuGu} and Theorem 1.11 of \cite{Gu2002}. Here we rewrite them in terms of our language as well as a proof for later applications and for the convenience of readers. We also would like to point out that the work of Burstall and Guest does  not consider initial conditions, e.g.  for extended frames. So in general the value of an extended  frame will
 be different from $e$ at some fixed basepoint $z_0$.  We will show below in Theorem 4.20 below that also the standard DPW theory with prescribed initial condition at some fixed base point produces all harmonic maps of finite uniton number as well. All statements in the theorems below can be derived from the work of Burstall and Guest. But the presentation uses substantially  the DPW method. Therefore, for the convenience of the reader, we include a proof.}

\begin{theorem} \label{thm-finite-uniton2}We retain the  {the notation and the assumptions  of  Theorem \ref{thm-finite-uniton0} and Theorem \ref{thm-finite-uniton1} as needed.}
\begin{enumerate}
  \item
  Let $\mathbb{F}:M\rightarrow G$ be a harmonic map of finite uniton number  {with extended solution $\Phi(z,\bar z,\lambda)$ as stated in Theorem \ref{thm-finite-uniton0}. Then ${\Phi_-}:=\gamma_{\xi}^{-1}\cdot \exp C \cdot \gamma_{\xi}=\gamma_{\xi}^{-1}\cdot\Phi\cdot\Phi_+$} \violet{has} a Maurer-Cartan form
 {\begin{equation}\label{eq-eta} {\mathbb A_-}:= {\Phi_-}^{-1}\mathrm{d} {\Phi_-}=\lambda^{-1}\sum_{0\leq j\leq r(\xi)-1} A_j'\dd z, \end{equation}
where each $A_j':M\rightarrow \mathfrak{g}^{\xi}_{j+1}$ is a meromorphic function  {on $M$} with poles in $D'$.}
Moreover, at some base point $z_0\in M\setminus D'$ we have
 \begin{equation}\label{eq-initial1}
 {\Phi_-}(z_0)= {\Phi_{-0}}:= {\gamma_{\xi}^{-1} \cdot \exp C(z_0) \cdot \gamma_{\xi}} \in \Lambda^- G^{\C},\ C(z_0)\in\mathfrak{u}^0_{\xi}.
 \end{equation}

 Conversely, given $ {\mathbb A_-}$ which takes values in  $\lambda^{-1}\cdot\sum_{0\leq j\leq r(\xi)-1}\mathfrak{g}^{\xi}_{j+1}$ and an initial condition of $ {\Phi_-}$ of the form \eqref{eq-initial1},  {assume that  on $ M$,
there exists a global meromorphic solution $ {\Phi_-}$  (with poles in $\D'$)} \violet{to}
  \begin{equation}\label{eq-A-1} {\Phi_-}^{-1}\mathrm{d} {\Phi_-}= {\mathbb A_-}, \ ~~~ {\Phi_-}(z_0)= {\Phi_{-0}}= {\gamma_{\xi}^{-1} \cdot \exp C_0 \cdot \gamma_{\xi}\ } \in \Lambda^- G^{\C}\hbox{ with } C_0\in\mathfrak{u}^0_{\xi}.\end{equation}
  {The Iwasawa decomposition of $\gamma_{\xi}\Phi_-$ gives a harmonic map $\mathbb{F}$ of finite uniton number in $G$.}

  {The above two procedures are inverse to each other when the initial conditions match.}
 \item Let $\mathcal{F}:M\rightarrow G/\hat{K}$ be a harmonic map of finite uniton number
  {with extended frame $F(z,\bar z,\lambda)$ based at $z_0 \in M$ with initial value $e$.} Embed $G/\hat{K}$ into $G$ as totally geodesic submanifold via the  { modified} Cartan embedding. Then there exists some canonical $\xi\in \mathcal{I}'$,
some discrete subset $D'\subset M$, such that
 \begin{equation}
G/\hat{K}\cong\{g(\exp \pi \xi)g^{-1}| g\in G\},
\end{equation}
and that
\begin{equation}\label{eq-F_-}
     {F_-= \gamma_{\xi}^{-1} \cdot \exp C \cdot \gamma_{\xi}}
\end{equation} is a meromorphic extended frame of $\mathcal{F}$ with the normalized potential having the form
 {\begin{equation}\eta=F_-^{-1}\mathrm{d}F_-=\lambda^{-1}\sum_{0\leq 2j\leq r(\xi)-1} A_{2j}'\dd z, \end{equation}
where $A_{2j}':M\rightarrow \mathfrak{g}^{\xi}_{2j+1}$ is a meromorphic function  on $M$ with poles in $D'$ for each $j$.}  And at the base point $z_0$
 \begin{equation}\label{eq-initial2}F_-(z_0)=F_{-0}:= {\gamma_{\xi}^{-1} \cdot \exp C(z_0) \cdot \gamma_{\xi} } \ in\ \Lambda^- G^{\C},\ C(z_0)\in(\mathfrak{u}^0_{\xi})_T.
 \end{equation}
 Conversely, given a  meromorphic normalized potential $\eta$ which takes values in  $\lambda^{-1}\cdot\sum_{0\leq 2j\leq r(\xi)-1}\mathfrak{g}^{\xi}_{2j+1}$ and an initial condition
  {$F_{-0}$ of $F_-$ of } the form \eqref{eq-initial2},
 {assume that on $M$ there exists a global  solution $F_-$ (with poles in $\D'$)}  {to}
 \begin{equation}\label{eq-initial3}F_-^{-1}\mathrm{d}F_-=\eta, \ ~~~F_-(z_0)=F_{-0}. \end{equation}
 {The Iwasawa decomposition of $F_-(z,\lambda)$ gives the extended frame of a harmonic} map of finite uniton number into $ \{g(\exp \pi \xi)g^{-1}| g\in G\}\cong G/\hat{K} $.

  {The above two procedures are inverse to each other when the initial conditions match.}
\end{enumerate}
\end{theorem}

 {Note that part (1) of Theorem \ref{thm-finite-uniton2} is essentially Proposition B1 of Appendix B of \cite{BuGu}.}

\begin{proof}
(1) First we note that by page 557 in \cite{BuGu},
\[ \gamma_{\xi}^{-1}X\gamma_{\xi}=\lambda^{-j-1}X,~ \hbox{ for any element } ~X\in\mathfrak{g}^{\xi}_{j+1}.\]
Together with the definition of $C$  { in Theorem \ref{thm-finite-uniton0},} we have
 { ${\gamma_{\xi}}^{-1} \cdot \exp C \cdot \gamma_{\xi} \in\Lambda^-G^{\mathbb{C}}$.}

Now consider the Maurer-Cartan form of  {$\gamma_{\xi}^{-1} \cdot \exp C \cdot \gamma_{\xi}$.} We have
\begin{equation*}\begin{split}( {\gamma_{\xi}^{-1} \cdot \exp C \cdot \gamma_{\xi})^{-1}\cdot }\mathrm{d}( {\gamma_{\xi}^{-1} \cdot \exp C \cdot \gamma_{\xi}})&=( {\gamma_{\xi}^{-1} \cdot
\exp C \cdot \gamma_{\xi}})^{-1}( {\gamma_{\xi}^{-1} \cdot \exp C \cdot \gamma_{\xi}})_z\mathrm{d}z \\
&=
 \gamma_{\xi}^{-1}\left( \exp C^{-1}(\exp C)_z\right) \gamma_{\xi}\mathrm{d}z,\end{split}\end{equation*}
since $\gamma_{\xi}$ is independent of $z$.

By \eqref{eq-C} in Theorem \ref{thm-finite-uniton0},
\begin{equation*}\label{eq-exp-C} (\exp C)^{-1}(\exp C)_z=\sum_{j=0}^{r(\xi)-1}\lambda^j A'_j \ \ \hbox{ with }\ \ \ A'_j:M \rightarrow \mathfrak{g}^{\xi}_{j+1},\ 0\leq j\leq r(\xi)-1.\end{equation*}
So
\begin{equation*}( {\gamma_{\xi}^{-1} \cdot \exp C \cdot \gamma_{\xi}})^{-1}
( {\gamma_{\xi}^{-1}\cdot \exp C \cdot \gamma_{\xi}})_z=\gamma_{\xi}^{-1} \left( \sum_{1\leq j\leq r(\xi)-1}\lambda^j A'_j\right)   \gamma_{\xi} =\sum_{1\leq j\leq r(\xi)-1} \lambda^j(\gamma_{\xi}^{-1} A'_j  \gamma_{\xi}).
\end{equation*}
By page 557 in \cite{BuGu}, for any element $X\in\mathfrak{g}^{\xi}_{j+1}$, $ \gamma_{\xi}^{-1}X\gamma_{\xi}=\lambda^{-j-1}X$.  Since $ A'_j $ takes values in $\mathfrak{g}^{\xi}_{j+1}$, we have
\[\gamma_{\xi}^{-1} A'_j  \gamma_{\xi}=\lambda^{-j-1}A'_j.\]
As a consequence,
\[( {\gamma_{\xi}^{-1}\cdot \exp C \cdot \gamma_{\xi}})^{-1}( {\gamma_{\xi}^{-1} \cdot
\exp C \cdot \gamma_{\xi}})_z=\sum_{1\leq j\leq r(\xi)-1} \lambda^j(\gamma_{\xi}^{-1} A'_j  \gamma_{\xi})= \lambda^{-1}\sum_{1\leq j\leq r(\xi)-1} A'_j .\]

 Conversely, let  $\Phi_-$ be a solution to \eqref{eq-A-1} as assumed. Then  $\Phi_-$ takes values in
$\Omega_{alg}G$, since $\lambda\mathbb A_-(\frac{\partial}{\partial z})$ takes values in a nilpotent Lie algebra.
{Letting $ \gamma_{\xi}\Phi_- =\Phi (\Phi_+)^{-1}$ } be the Iwasawa decomposition of $\gamma_{\xi}\Phi_-$. It hence produces a harmonic map of finite uniton number.

 (2) First one needs to restrict the above results to the case $A'_{2j+1}=0$ for all $j$ by  Theorem \ref{thm-finite-uniton1}.  By $(1)$ of Theorem \ref{thm-finite-uniton1} we can choose the extended solution associated to the harmonic map $\mathcal{F}$ such that
  $\Phi=\gamma_{\tilde\xi}(\lambda)F(z,\bar{z}, \lambda)h F(z,\bar{z}, 1)^{-1}$ holds. We see that in this case the Iwasawa decomposition of $\Phi$ yields
  $ \Phi_-=\Phi\Phi_+$, whence
   { \[\Phi_-=\Phi\Phi_+=\gamma_{\tilde\xi}(\lambda)F(z,\bar{z}, \lambda)hF(z,\bar{z}, 1)^{-1}\Phi_+,\]
 where the second equality follows  from (\ref{eq-Phi-F}).}
 {Consider $\tilde F(z,\bar z,\lambda):=\gamma_{\xi}(\lambda)^{-1}\gamma_{\tilde\xi}(\lambda)F(z,\bar{z}, \lambda)$. It is also an extended frame of the harmonic map $\mathcal F(z,\bar z,\lambda=1)$, but at the basepoint $z_0$ it might have an initial condition different from the initial condition of $F(z,\bar z, \lambda)$:  $\tilde F(z_0,\bar z_0,\lambda)=\gamma_{\xi}(\lambda)^{-1}\gamma_{\tilde\xi}(\lambda)F(z_0,\bar{z}_0, \lambda)$.} Also note that $\tilde F(z,\bar z,1)=F(z,\bar{z}, 1)$. We can w.l.g. replace $F(z,\bar z,\lambda)$ by $\tilde F(z,\bar z,\lambda)$, since we do not assume any special initial conditions. Then we have
\[\Phi_-=F(z,\bar{z}, \lambda)h F(z,\bar{z}, 1)^{-1}\Phi_+=F_-(z,\lambda)F_+(z,\bar{z}, \lambda)h F(z,\bar{z}, 1)^{-1}\Phi_+.\]
So $F_-(z,\lambda)=\Phi_-$. By (1) of this Theorem, we finish the proof of (2).
 \end{proof}

\begin{remark}\
\begin{enumerate}
\item Note that $\mathbb{A}_-$ in \eqref{eq-initial1} is not the usual normalized potential, since the harmonic map is mapped into the Lie group $G$ instead of $G/K$ and we do not have the initial condition $e$. Also by the discussion in Section 4.1.1 and for convenience we still call it ``normalized potential". For more relations between $\mathbb{A}_-$ and the real normalized potential (embedding $G$ to construct a harmonic map into $G\times G/G$), we also refer the readers to \cite{Do-Es}.

\item The initial condition  {$\Phi_{-0}$ and $F_{-0}$ respectively} of Theorem \ref{thm-finite-uniton2} can be removed by using dressing (see Theorem 1.11 of \cite{Gu2002}). For instance, assume that
$\hat{F}_-^{-1}\mathrm{d}\hat{F}_-=\eta, \  {\hat{F}_-(z_0,\lambda)}=e$. Then $F_-=F_{-0}\hat{F}_-$. By Iwasawa splitting we have
\[F_-=FF_+,\ \   \hat{F}_-=\hat{F}\hat{F}_+.\ \]
Assume that $F_{-0}=\gamma_0\gamma_+$ with $\gamma_0\in\Lambda G,$ $\gamma_+\in\Lambda^+ G^{\mathbb{C}}$. Therefore we have
\[FF_+=\gamma_0\gamma_+\hat{F}\hat{F}_+.\]
As a consequence, we obtain
\begin{equation*} \gamma_+\hat{F}=\gamma_0^{-1}FF_+\hat{F}_+^{-1}.\end{equation*}
Hence
\begin{equation}\gamma_+\sharp\hat{F}=\gamma_0^{-1}F.\end{equation}
So up to a rigid motion $\gamma_0^{-1}$, $F$ is the dressing of $\hat{F}$ by $\gamma_+$ (compare with Corollary \ref{cor-finite-dress}).
\end{enumerate}
  \end{remark}

\subsection{On the initial conditions of normalized meromorphic frames: compact case}

In the last theorem we have given a precise relation between certain ``meromorphic potentials" and harmonic maps of finite uniton type. For this one needs very specific initial conditions and  { very specific dressing matrices} respectively.
 { In the case of the last theorem the initial conditions occurring are very  complicated and very difficult to produce. So it is important and useful to show that all harmonic maps of finite uniton type into compact inner symmetric spaces can be determined by the data as given, but with initial condition $e$, which is the goal of this subsection}.

To this end, we need some preparations.
First, for an arbitrary element  {$\xi \in \mathcal{I}'$} we have the following decompositions
 {(see e.g. \eqref{defgjxi}):}
\begin{equation}\label{eq-pq-decom-g}
  { \mathfrak{g}^{\C}=\sum_j  \mathfrak{g}^{\xi}_j=\left(\sum_{j\geq0} \mathfrak{g}^{\xi}_j\right)\oplus\left(\sum_{j<0} \mathfrak{g}^{\xi}_j\right)=\mathfrak{pr}\oplus\mathfrak{q}.}
\end{equation}
On the Lie group level, let $\mathds{PR}$ be the  {complex connected} Lie subgroup of $G^{\C}$ with Lie algebra
$\mathfrak{pr}$ and  let $\mathds{Q}$ denote the  {complex connected} Lie subgroup of $G^{\C}$ with Lie algebra $\mathfrak{q}$.
Let $\mathfrak{W}$ denote the Weyl group of $G$. Then we have the decomposition
\begin{equation}\label{eq-pq-decom-G}
    G^{\C}=\bigcup_{\omega\in \mathfrak{W_{\xi}}}\mathds{PR}\cdot\omega\cdot\mathds{Q}.
\end{equation}
where $\mathfrak{W_{\xi}}$ is  {some quotient of $\mathfrak{W}$.} This follows from Corollary 3.2.3 of \cite{Do-Gr-Sz}.

\begin{theorem} \label{thm-PR-Q-decomposition}
 {The set $\mathds{PR} \cdot \mathds{Q} \subset G^{\C}$ obtained by pointwise multiplication is open in $G^\C$ and the pointwise multiplication map $ \mathds{PR}\times\mathds{Q} \longrightarrow  \mathds{PR}\cdot\mathds{Q}$ is biholomorphic.}
 \end{theorem}

\begin{proof} By Theorem 2.4.1 (a) of \cite{Do-Gr-Sz}, $ \mathds{PR}\cdot\mathds{Q}$ is open in $G^{\C}$. By Theorem 2.4.1 (b) of \cite{Do-Gr-Sz},
  $ \mathds{PR}\times\mathds{Q}\longrightarrow  \mathds{PR}\cdot\mathds{Q}$ is a holomorphic diffeomorphism.
\end{proof}

We also need the following  { lemma:}
 \begin{lemma} \label{lemma-}  {Let $\xi \in \mathcal{I}'$ so that $h=\exp(\pi\xi)$. Then we obtain}
 \begin{equation}\label{eq-pr-p}
    \mathfrak{pr}\cap \mathfrak{p}^{\C}= \sum_{j\geq0}g^{\xi}_{2j+1}.
 \end{equation}
 \end{lemma}

\begin{proof}
 {Recall that $\mathfrak{g}=\mathfrak{k}\oplus\mathfrak{p}$
	with $\mathfrak{k}=Lie(K)$, and $\mathfrak{g}^{\C}=\mathfrak{k}^{\C}\oplus\mathfrak{p}^{\C}$.}
We have that $\mathfrak{k}=\{~X\in\mathfrak{g}~|~hX=Xh\ \},~~\mathfrak{k}^{\C}=\{~X\in\mathfrak{g}^{\C}~|~hX=Xh\ \}.$
For any element $X$ in $g^{\xi}_{2j}$,
\[hXh^{-1}=\exp( \pi\xi)\cdot X\cdot\exp(-\pi\xi)=\exp(\pi \cdot ad\xi) X=e^{2 j\pi  \sqrt{-1} }X=X.\]
Similarly, for any element $X$ in $g^{\xi}_{2j+1}$,
\[hXh^{-1}=\exp( \pi\xi)\cdot X\cdot\exp(-\pi\xi)=\exp( \pi\cdot ad\xi) X=e^{(2j+1)\pi \sqrt{-1}}X=-X.\]
\end{proof}

 {Next we will apply Theorem \ref{thm-finite-uniton2} and Theorem \ref{thm-PR-Q-decomposition} to show that for a harmonic map $\mathcal F$ an extended frame with initial condition $e$, Theorem \ref{thm-finite-uniton2} still holds.}

 \begin{theorem} \label{thm-finite-uniton-in}
Let G be a connected, semisimple, compact Lie group with trivial center. Let $\mathcal{F}:M\rightarrow G/\hat{K}$ be a harmonic map of
 finite uniton type  into the compact inner symmetric space $G/\hat{K}$.  {Let $z_0 \in M$ and $F(z, \bar z, \lambda)$ an extended frame of $\mathcal{F}$ satisfying
  $F(z_0, \bar z_0, \lambda) = e$ for all $\lambda \in S^1$.  Then there exists some canonical $\xi_{can}\in \mathcal{I}'$,
some discrete subset $D'\subset M$, such that
$\mathfrak{C}_h(G/\hat{K})=\{ghg^{-1}| g\in G\}$ with $h=\exp (\pi \xi_{can})$,
and  the normalized potential of $\mathcal{F}$ has the form
\begin{equation}\eta=F_-^{-1}\mathrm{d}F_-=\lambda^{-1}\sum_{0\leq 2j\leq r(\xi_{can})-1} A_{2j}'\mathrm{d}z, \end{equation}}
 {where $A_{2j}':M\rightarrow \mathfrak{g}^{\xi_{can}}_{2j+1}$ is a meromorphic function with poles in $D'$ for each $j$}. In particular, the normalized potential of $ \mathcal{F}$ is contained in a nilpotent {Lie algebra
 and  it is invariant} under the fundamental group of $M$.
 \end{theorem}
 \begin{proof} {
 By  Proposition \ref{eq-check-xi}, there exists some $\tilde\xi\in\mathcal I'$ such that
$\Phi=\gamma_{\tilde\xi} F(z,\bar{z},  \lambda)F(z,\bar{z}, 1)^{-1}$ is  an extended solution  of $\mathbb F(z,\bar z,1)=F(z, \bar z, 1)hF(z, \bar z, 1)^{-1}$  which is ${\T}-$invariant and satisfies $\exp(\pi\tilde\xi)=h$. Then by Proposition 4.1 of \cite{BuGu}, there exists some $\hat\xi\in\mathcal I'$ which may not be canonical, and some discrete subset $D'\subset M$ such that on $M\backslash D'$ we have
\[\gamma^{-1}_{\hat\xi}\Phi=\gamma^{-1}_{\hat\xi} \cdot \exp C \cdot \gamma_{\hat\xi}\cdot W_+(z,\bar{z},\lambda)=\hat{F}_- (z,\lambda) W_+(z,\bar{z},\lambda), \text{where } \hat{F}_-(z,\lambda)= \gamma^{-1}_{\hat\xi}
\cdot \exp C \cdot \gamma_{\hat\xi},\]
with $C:M\rightarrow(\mathfrak{u}^0_{\hat\xi})_{\T}$ being meromorphic with poles in $D'$ and
{$W_+: M \rightarrow \Lambda^{+}G^{\C}$.} Note that {$\gamma_{\hat\xi}\in(\Omega_{alg}G)_{\T}$}
since both $\Phi$ and $\gamma_{\tilde\xi} F(z,\bar{z},  \lambda)F(z,\bar{z}, 1)^{-1}$ are contained in $(\Omega_{alg}G)_{\T}$. In particular, $\exp(\pi\hat\xi)=h$ and both of $\gamma_{\tilde\xi}^{-1}$ and $\gamma_{\hat\xi}$  take values in $\mathds{PR}$,  {where we define $\mathds{PR}$ by $\hat\xi$}. }

 {Let \[\hat{F}(z,\bar{z},\lambda)=\gamma_{\hat\xi}^{-1}\gamma_{\tilde\xi}F(z,\bar{z},\lambda).\]
Since all of $\gamma_{\xi}$, $\gamma_{\tilde\xi}$ and $F(z,\bar{z},\lambda)$ are $\sigma-$twisted, we see that $\hat{F}(z,\lambda)$ is $\sigma-$twisted and hence is an extended frame of $\mathcal F$ with different initial condition  {$\hat{F}(z_0,\bar{z}_0,\lambda)=\gamma_{\hat\xi}(\lambda)^{-1}\gamma_{\tilde\xi}(\lambda)$.}
Moreover, since $ \hat{F}_-(z,\lambda)=  {\gamma^{-1}_{\hat\xi} \cdot \exp C \cdot \gamma_{\hat\xi}}$ takes values in $\Lambda^-G^{\mathbb C}_{\sigma}$, we have the Birkhoff decomposition of $\hat F(z,\bar{z},\lambda)$:
\[\hat{F}=\gamma^{-1}_{\xi}\Phi(z,\bar{z},\lambda) F(z,\bar{z},1)=\hat{F}_-(z,\lambda) \hat{ V}_+(z,\bar{z},\lambda) \ \hbox{ with }\ \hat{ V}_+(z,\bar{z},\lambda)=W_+(z,\bar{z},\lambda)F(z,\bar{z},1).\]
Now near $z_0$, we also have the Birkhoff decomposition of $F(z,\bar{z},\lambda)$:
\[F(z,\bar{z},\lambda)=F_-(z,\lambda)V_+
=\gamma_{\tilde\xi}^{-1}\gamma_{\hat\xi}\hat{F}(z,\bar{z},\lambda)
=\gamma_{\tilde\xi}^{-1}\gamma_{\hat\xi}\hat{F}_-(z,\lambda)\hat{ V}_+(z,\bar{z},\lambda).\]
Decompose $\hat{ V}_+(z,\bar{z},\lambda)$ according to \eqref{eq-pq-decom-G}:
\[\hat{ V}_+(z,\bar{z},\lambda)=\mathbf{R}\omega \mathbf{Q}~~ \hbox{ with }~ \mathbf{R}\in\mathds{PR} \hbox{ and } \mathbf{Q}\in\mathds{Q}.\]
By Theorem \ref{thm-PR-Q-decomposition}, since $\hat V_+$ is holomorphic in $\lambda$ for all $\lambda\in\mathbb C$, $\mathbf{R}$ and  $\mathbf{Q}$ are also holomorphic in $\lambda$ for all $\lambda\in\mathbb C$, i.e., $\mathbf{R}, \mathbf{Q}\in\Lambda^+G^{\C}$.}

 {Since  $F(z,\bar{z},\lambda)\rightarrow e$ as $z\rightarrow z_0$, we obtain \[\gamma_{\tilde\xi}^{-1}\gamma_{\hat\xi}\hat{F}_-(z,\lambda) \mathbf{R}\omega \mathbf{Q}\rightarrow e, \hbox{ if $z\rightarrow z_0$}.\] Note that all of
$\gamma_{\tilde\xi}^{-1}$, $\gamma_{\hat\xi}$ and $\hat{F}_-(z,\lambda)$ take values in $\mathds{PR}$. Since $\mathds{PR}\cdot\mathds{Q}$ is open and $e\in\mathds{PR}\cdot\mathds{Q}$, $\omega=e$  near $z_0$. As a consequence, when $z\rightarrow z_0$,  $\gamma_{\tilde\xi}^{-1}\gamma_{\hat\xi}\hat{F}_-(z,\lambda) \mathbf{R}\rightarrow e$ and $\mathbf{Q}\rightarrow e$. Moreover, considering $F_-$, we obtain
\[F_-=(\gamma_{\tilde\xi}^{-1}\gamma_{\hat\xi}\hat{F}_-(z,\lambda) \mathbf{R})_-.\]
Since $\mathbf{R}\in\mathds {PR}$ and all of
$\gamma_{\tilde\xi}^{-1}$, $\gamma_{\hat\xi}$ and $\hat{F}_-(z,\lambda)$ also take values in $\mathds{PR}$, we see that $F_-$ also takes values in $\mathds{PR}$.
Consider \[\eta_-=F_-^{-1}\mathrm{d}F_-=\lambda^{-1}\eta_{-1}\mathrm{d}z.\] Since $F_-$ takes values in $\mathds{PR}$, $\eta_{-1}$ takes values in $\mathfrak{pr}$. On the other hand, $\eta_{-1}$ also takes values in $\mathfrak{p}^{\C}$. In a sum, by \eqref{eq-pr-p}, $\eta_{-1}$ takes values in \[\mathfrak{pr}\cap\mathfrak{p}^{\C}=\sum_{0\leq 2j\leq r(\hat\xi)-1} \mathfrak{g}^{\hat\xi}_{2j+1}.\]  Let $\xi_{can}$ be the canonical element derived from $\hat\xi$ as in Lemma \ref{lemma-xi}. By Lemma \ref{lemma-xi} we have
\[\sum_{0\leq j\leq r(\hat\xi)-1} \mathfrak{g}^{\hat\xi}_{j+1}=\sum_{0\leq j\leq r(\xi_{can})-1}  \mathfrak{g}^{\xi_{can}}_{j+1}.\]
From the proof of Theorem 5.4 of \cite{BuGu}, we also obtain
\[ \sum_{0\leq 2j\leq r(\hat\xi)-1} \mathfrak{g}^{\hat\xi}_{2j+1}=\sum_{0\leq 2j\leq r(\xi_{can})-1}  \mathfrak{g}^{\xi_{can}}_{2j+1}.\]
 So   $\eta_{-1}$ takes values in $\sum_{0\leq 2j\leq r(\xi_{can})-1}  \mathfrak{g}^{\xi_{can}}_{2j+1}$, which is contained in the same nilpotential Lie subalgebra as $\hat\eta_{-1}=\hat F_{-}^{-1}(\hat F_{-})_{z}$. Since $F(z_0,\bar{z}_0,\lambda)=e$, we have $F_{-}(z_0,\lambda)=e$ as well. }
  \end{proof}

\section{Harmonic maps of finite uniton type into non-compact inner symmetric spaces}

 {In the last section we have shown how one can relate the construction schemes of Burstall-Guest and DPW respectively of harmonic maps into compact inner symmetric spaces into each other. In this section, we will show that the DPW interpretation of the theory of Burstall and Guest \cite{BuGu} also holds for harmonic maps of finite uniton type into non-compact inner symmetric spaces. We will also review briefly the application of this theory to the coarse} classification of Willmore surfaces \cite{Wang-2}.

\subsection{ {The case of non-compact inner symmetric spaces}}
 {We will apply the results of  Section 3.3. Let $G/K$ is a non-compact inner symmetric space} and $U$  a maximal compact subgroup of the complexification $G^\C$ of $G$ which is left invariant by the natural real form involution of $G$ and the extension of $\sigma$ to $G^\C$. Combining Theorem \ref{thm-noncompact}, Theorem \ref{thm-finite-uniton2},   {Theorem \ref{thm-finite-uniton-in}} and Corollary \ref{cor-finite}, we obtain
\begin{theorem}\label{thm-finite-uniton-n-com}
Let $\mathcal{F}: \tilde{M}\rightarrow G/K$  be a harmonic map of finite uniton type, and
 ${\mathcal{F}_ {U}:} \tilde{M} \rightarrow U/(U\cap K^{\C})$ the compact dual harmonic map of $\mathcal{F}$, with base point $z_0\in\tilde{M}$  {such that $\mathcal{F}_{z=z_0}=eK$ and $(\mathcal{F}_{U})_{z=z_0}=e(U\cap K^{\C})$.}  {Then the normalized potential $\eta_-$ and the normalized extended framing  $F_-$ derived from  Theorem \ref{thm-finite-uniton-in} for
 $\mathcal{F}_U$,  provides also a normalized potential $\eta_-$ and a normalized extended framing $F_-$  of $\mathcal{F}$, with initial condition $F_-(z_0,\lambda)=e$.} {In particular, all harmonic maps of finite uniton type $\mathcal{F}: \tilde{M}\rightarrow G/K$  can be obtained in this way.}
\end{theorem}

 {\begin{remark}\
\begin{enumerate}
\item By Theorem \ref{thm-finite-uniton-n-com}, on the (nilpotent) normalized potential level, one can classify  all harmonic maps $\mathcal{F}: \tilde{M}\rightarrow G/K$ of finite uniton type by classifying  all harmonic maps of finite uniton type from $\tilde{M}$ to $U/(U\cap K^{\C})$.
\item Note that the DPW method yields a 1-1 relation between normalized potentials
and harmonic maps if one chooses a base point and an initial condition $e$ for all relevant extended frames.
\end{enumerate}
\end{remark}}
\subsection{Nilpotent normalized potentials of Willmore surfaces of finite uniton type}

We will end this paper with a  {brief view} of a coarse classification and a construction of new Willmore surfaces of finite uniton type, in the spirit of Section 4. To be concrete, we  give a coarse classification of Willmore two-spheres in $S^{n+2}$ by classifying all the possible nilpotent Lie sub-algebras related to the corresponding harmonic conformal Gauss maps, see Theorem 3.1, Theorem 3.3 of \cite{Wang-1}. This classification indicates that Willmore two spheres may inherit more complicated and new geometric structures.  Moreover, by concrete computations of Iwasawa decompositions, we construct new Willmore two-spheres (Section 3.1 and Section 3.2 of \cite{Wang-3}).

To state the coarse classification and constructions, we first recall that a Willmore surface $y:M\rightarrow S^{n+2}$ is globally related to a harmonic conformal Gauss map  $\mathcal{F}:M\rightarrow SO^+(1,n+3)/SO^+(1,3)\times SO(n)$ satisfying some isotropy condition (\cite{Bryant1984}, \cite{Ejiri1988}, \cite{DoWa11}, \cite{Wang-1}). Here
$SO^+(1,n+3)=SO(1,n+3)_0$ is the connected subgroup of
\begin{equation*}SO(1,n+3):=\{A\in Mat(n+4,\mathbb{C})\ |\ A^tI_{1,n+3}A=I_{1,n+3}, \det A=1,\ A=\bar{A}
\},
\end{equation*}
and the subgroup $K=SO^+(1,3)\times SO(n)\subset SO^+(1,n+3)$ is defined by the involution
 \begin{equation}\begin{array}{ll}
\sigma:    SO^+(1,n+3)&\rightarrow SO^+(1,n+3)\\
 \ \ \ \ \ \ \ A &\mapsto DAD^{-1}.
\end{array}\end{equation}
with $D=diag\{-I_4,I_n\}.$ Moreover
\begin{equation*} \mathfrak{so}(1,n+3):=\{A\in Mat(n+4,\mathbb{R})\ |\ A^tI_{1,n+3}+I_{1,n+3}A\}.
\end{equation*}
Let $\mathfrak{k}$ be the Lie algebra of $SO^+(1,3)\times SO(n)\subset SO^+(1,n+3)$ and $\mathfrak{so}(1,n+3)=\mathfrak{k}\oplus\mathfrak{p}$. Then
\[\mathfrak{k}=\left\{\left(
                               \begin{array}{cc}
                                 A_1 & 0 \\
                                 0 & A_2 \\
                               \end{array}
                             \right)| A_1\in Mat(4,\mathbb{R}),\ A_2\in Mat(n,\mathbb{R}),\ A_1^tI_{1,3}+I_{1,3}A_1=0,\ A_2^t+A_2=0
\right\}\]
and
\[\mathfrak{p}=\left\{\left(
                               \begin{array}{cc}
                                0 & B_1 \\
                                 -B_1^tI_{1,3} &0 \\
                               \end{array}
                             \right)| B_1\in Mat(4\times n,\mathbb{R})
\right\}.\]
 {It is straightforward to see that the compact dual of $SO^+(1,n+3)$ and $SO^+(1,3)\times SO(n)$ are $SO(n+4)$ and $SO(4)\times SO(n)$ respectively.}

 Following \cite{DoWa11} that we call a conformally harmonic map  $\mathcal{F}:M\rightarrow SO^+(1,n+3)/SO^+(1,3)\times SO(n)$ a strongly conformally \footnote{ {Note that in \cite{BW} the notion of ``strongly conformal" has been used in another sense.}} harmonic map if the $\mathfrak{p}^{\C}-$part of its Maurer-Cartan form,
\[\left(
                     \begin{array}{cc}
                       0 & {B}_1 \\
                       -{B}^{t}_1I_{1,3} & 0 \\
                     \end{array}
                   \right)\mathrm{d}z,\]
                   satisfies the isotropy condition
                   \begin{equation}\label{eq-isotropic}
                    B_1^tI_{1,3}B_1=0.
                   \end{equation}

                   Now assume $n+4=2m$ so that $SO^+(2m)/SO^+(1,3)\times SO(2m)$ is an inner symmetric space.
\begin{theorem}(Theorem 3.1 of \cite{Wang-1})
Let $\mathcal{F}:M\rightarrow SO^+(1,n+3)/SO^+(1,3)\times SO(n)$ be a finite-uniton type strongly conformally harmonic map, with $n+4=2m$. Let
$$\eta=\lambda^{-1}\left(
                     \begin{array}{cc}
                       0 & \hat{B}_1 \\
                       -\hat{B}^{t}_1I_{1,3} & 0 \\
                     \end{array}
                   \right)\mathrm{d}z $$
  be the normalized potential of $\mathcal{F}$, with $$\hat{B}_1=(\hat{B}_{11},\cdots,\hat{B}_{1,m-2}),\ \hat{B}_{1j}=(\mathrm{v}_{j },\hat{ \mathrm{v}}_{j })\in Mat(4\times 2,\mathbb{C}).$$
 Then up to some conjugation by a constant matrix, every $\hat{B}_{j1}$ of $\hat{B}_1$ has one of the two forms:
\begin{equation}(i) \ \mathrm{v}_{j}= \left(
                                          \begin{array}{ccccccc}
                                            h_{1j}    \\
                                            h_{1j}  \\
                                            h_{3j} \\
                                            ih_{3j}  \\
                                          \end{array}
                                        \right),\ \hat{ \mathrm{v}}_{j } =
                                        \left(
                                          \begin{array}{ccccccc}
                                            \hat{h}_{1j}    \\
                                            \hat{h}_{1j}   \\
                                            \hat{h}_{3j} \\
                                            i\hat{h}_{3j}   \\
                                          \end{array}
                                        \right);\
 (ii) \ \mathrm{v}_{j}= \left(
                                          \begin{array}{ccccccc}
                                            h_{1j}  \\
                                            h_{2j}    \\
                                            h_{3j}  \\
                                            h_{4j}  \\
                                          \end{array}
                                        \right),\ \ \hat{ \mathrm{v}}_{j } = i \mathrm{v}_{j}= \left(
                                          \begin{array}{ccccccc}
                                            ih_{1j}  \\
                                            ih_{2j}    \\
                                            ih_{3j}  \\
                                            ih_{4j}  \\
                                          \end{array}
                                        \right).
\end{equation}
And all of $\{\mathrm{v}_j,\  \hat{ \mathrm{v}}_{j }\}$ satisfy the equations
$$\mathrm{v}_j^tI_{1,3}\mathrm{v}_l=\mathrm{v}_j^tI_{1,3}\hat{\mathrm{v}}_l=\hat{\mathrm{v}}_j^tI_{1,3}\hat{\mathrm{v}}_l=0, \ j,l=1,\cdots,m-2.$$
In other words, there are $m-1$ types of normalized potentials with $\hat{B}_{1}$ satisfying $\hat{B}_{1 }^tI_{1,3}\hat{B}_{1 }=0$, namely those being of one of the following $m-1$ forms (up to some conjugation):

$(1)$ (all pairs are of type (i))

\begin{equation}
                         \hat{B}_{1}= \left(
                                          \begin{array}{ccccccc}
                                            h_{11} & \hat{h}_{11} &  h_{12} & \hat{h}_{12} &\cdots &  h_{1,m-2}& \hat{h}_{1,m-2} \\
                                             h_{11} & \hat{h}_{11} &  h_{12}& \hat{h}_{12}&\cdots &  h_{1,m-2} & \hat{h}_{1,m-2}\\
                                            h_{31}& \hat{h}_{31} &  h_{32}& \hat{h}_{32} &\cdots &  h_{3,m-2}& \hat{h}_{3,m-2} \\
                                            ih_{31}& i\hat{h}_{31} &  ih_{32}& i\hat{h}_{32}&\cdots &  ih_{3,m-2}& i\hat{h}_{3,m-2} \\
                                          \end{array}
                                        \right);\end{equation}

$(2)$ (the first pair is of type (ii), all others are of type (i))
\begin{equation}
                           \hat{B}_{1}= \left(
                                          \begin{array}{ccccccc}
                                            h_{11} & i {h}_{11} &  h_{12} & \hat{h}_{12} &\cdots &  h_{1,m-2}& \hat{h}_{1,m-2} \\
                                             h_{21} & i {h}_{21} &  h_{12}& \hat{h}_{12}&\cdots &  h_{1,m-2} & \hat{h}_{1,m-2}\\
                                            h_{31}& i {h}_{31} &  h_{32}& \hat{h}_{32} &\cdots &  h_{3,m-2}& \hat{h}_{3,m-2} \\
                                            h_{41}& i {h}_{41} &  ih_{32}& i\hat{h}_{32}&\cdots &  ih_{3,m-2}& i\hat{h}_{3,m-2} \\
                                          \end{array}
                                        \right);\end{equation}
Introducing consecutively  more pairs of type (ii), one finally arrives at

       $(m-1)$ (all pairs are of type (ii))
\begin{equation}
                           \hat{B}_{1}= \left(
                                          \begin{array}{ccccccc}
                                            h_{11} & i {h}_{11} &  h_{12} & i {h}_{12} &\cdots &  h_{1,m-2}& i{h}_{1,m-2} \\
                                             h_{21} & i {h}_{21} &  h_{22}& i {h}_{22}&\cdots &  h_{2,m-2} & i{h}_{2,m-2}\\
                                            h_{31}& i {h}_{31} &  h_{32} & i {h}_{32} &\cdots &  h_{3,m-2}& i{h}_{3,m-2} \\
                                            h_{41}& i {h}_{41} &  h_{42} & i {h}_{42}&\cdots &   h_{4,m-2}& ih_{4,m-2} \\
                                          \end{array}
                                        \right).\end{equation}
                            {Here $r(f)\leq 2$ for Case $(1)$ and Case $(m-1)$. For other cases, $r(f)\leq 3,4,5,6$ or $8$, depending on the structure of the potential.}

\end{theorem}
 {This theorem follows from the classification of nilpotent Lie sub-algebras related to the symmetric space  $SO^+(1,n+3)/SO^+(1,3)\times SO(n)$, together with a restriction  of the isotropy condition \eqref{eq-isotropic} on potentials (to derive Willmore surfaces). See \cite{Wang-1} for more details.}
\begin{example}(\cite{DoWa12}, \cite{Wang-3})\label{example}
 Let \[\eta=\lambda^{-1}\left(
                      \begin{array}{cc}
                        0 & \hat{B}_1 \\
                        -\hat{B}_1^tI_{1,3} & 0 \\
                      \end{array}
                    \right)\mathrm{d}z,\ ~ \hbox{ with } ~\ \hat{B}_1=\frac{1}{2}\left(
                     \begin{array}{cccc}
                       2iz&  -2z & -i & 1 \\
                       -2iz&  2z & -i & 1 \\
                       -2 & -2i & -z & -iz  \\
                       2i & -2 & -iz & z  \\
                     \end{array}
                   \right).\]
 {Each extended frame $F(z,\bar z,\lambda)$ derived  from this potential has singularities, while the corresponding harmonic maps and  the corresponding  Willmore two-spheres are globally well-defined.}
The associated family of Willmore two-spheres $x_{\lambda}$, $\lambda\in S^1$, corresponding to $\eta$, is
\begin{equation}\label{example1}
\begin{split}x_{\lambda}&=\frac{1}{ \left(1+r^2+\frac{5r^4}{4}+\frac{4r^6}{9}+\frac{r^8}{36}\right)}
\left(
                          \begin{array}{c}
                            \left(1-r^2-\frac{3r^4}{4}+\frac{4r^6}{9}-\frac{r^8}{36}\right) \\
                            -i\left(z- \bar{z})(1+\frac{r^6}{9})\right) \\
                            \left(z+\bar{z})(1+\frac{r^6}{9})\right) \\
                            -i\left((\lambda^{-1}z^2-\lambda \bar{z}^2)(1-\frac{r^4}{12})\right) \\
                            \left((\lambda^{-1}z^2+\lambda \bar{z}^2)(1-\frac{r^4}{12})\right) \\
                            -i\frac{r^2}{2}(\lambda^{-1}z-\lambda \bar{z})(1+\frac{4r^2}{3}) \\
                            \frac{r^2}{2} (\lambda^{-1}z+\lambda \bar{z})(1+\frac{4r^2}{3})  \\
                          \end{array}
                        \right)\\
  \end{split}
\end{equation}
 $x_{\lambda}:S^2\rightarrow S^6$ is a Willmore immersion in $S^6$, which is non S-Willmore, full, and totally isotropic (See \cite{Ejiri1988, BR} for a twistor method for totally isotropic surfaces). In particular, $x_\lambda$ does not have any branch points. The uniton number of $x_{\lambda}$ is $2$ and therefore its conformal Gauss map is $S^1-$invariant by Corollary 5.6 of \cite{BuGu}.
\end{example}


\section{Appendix I: The loop group method for harmonic maps}
In the section we will outline the basic loop group method for harmonic maps in an inner symmetric space $G/K$ for readers' convenience. For full details we refer to \cite{DoWa-AT}.

  \subsection{Harmonic Maps into Symmetric Spaces}

\begin{itemize}
\item G connected semisimple real Lie group with trivial center.
\item {$\sigma(g) = h g h^{-1}$ inner automorphism of G of order 2}, with $h\in G$.

\item {$\hat{K} = Fix^\sigma (G)$ and $K$ is the connected component of $\hat{K}$ containing the identity $e$ of $G$}.
\item $M$ is a Riemann surface with $\pi: \tilde{M} \rightarrow M$ universal cover of $M$.
\item
We have the following commutative diagram
\begin{displaymath}
    \xymatrix{\tilde{M} \ar[r]^{F}  \ar[d]_{\pi_M}   \ar[rd]^{\tilde{\mathcal{F}}}   &  G \ar[d]^{\pi_S}\\
             M  \ar[r]_{\mathcal{F}}  & {G/{\hat{K}}} }
\end{displaymath}

\item $\pi_M$ and $\pi_S$  are natural projections
\item $\mathcal{F}$ is a smooth map (will always denote a harmonic map below)
\item $\tilde{\mathcal{F}}$ is the natural lift of $\mathcal{F}$ to $\tilde{M}$.
\item F is simultaneously the moving frame of $\mathcal{F}$ and $\tilde{\mathcal{F}}$, which may have two singular points when $M = S^2$ \cite{DoWa11}.
\item  We put: $\mathfrak{g} = Lie (G)$ and decompose it relative to $\sigma$:
$\mathfrak{g}  = \mathfrak{k}  + \mathfrak{p}$.
\item   Let $\mathcal{F}: {\tilde{M}} \rightarrow G/{\hat{K}}$ be a smooth map and
    $F:\tilde{M} \rightarrow G$  a {frame for $\mathcal{F}.$}  Decompose the {Maurer-Cartan} form $\alpha = F^{-1} \dd F$ of $F$ in the form
  $$  \alpha = \alpha^{\prime}_\mathfrak{p} \dd z+ \alpha_\mathfrak{k} +
  \alpha^{\prime\prime}_\mathfrak{p} \dd\bar{z}$$
  with $\alpha^{\prime}_\mathfrak{p}$ a $(1,0)-$form and  $\alpha^{\prime\prime}_\mathfrak{p}$ a $(0,1)-$form
 and $\alpha_\mathfrak{k} $ a $\mathfrak{k}-$valued real $1-$form.
\item  {\bf Pohlmeyer's Theorem} \cite{DPW}:
\begin{theorem}  The map $\mathcal{F}$ is  {harmonic}  if and only if
 $\alpha_\lambda = \lambda^{-1} \alpha^{\prime}_\mathfrak{p} dz + \alpha_\mathfrak{k} +
 \lambda \alpha^{\prime\prime}_\mathfrak{p} d\bar{z}$  is integrable for all
 $\lambda \in S^1.$

 In this case:  $F\zl^{-1} \dd F\zl = \lambda^{-1} \alpha^\prime_\mathfrak{p} \dd z + \alpha_\mathfrak{k} + \lambda \alpha^{\prime\prime}_\mathfrak{p} \dd\bar{z}$  \hspace{2mm}  is integrable! We call $F\zl$ an extended frame based at $z_0\in M$ if $F\zl|_{z=z_0}=e$.
 \end{theorem}
\item The monodromy matrix $\chi(\gamma,\lambda)$:
\begin{theorem}
{Let $\mathcal{F}: M \rightarrow G/{\hat{K}}$ be a harmonic map and $z_0$  a fixed basepoint.} {Let,  as above, $\tilde{\mathcal{F}}: \tilde{M} \rightarrow G/{\hat{K}}$ denote the natural lift of $\mathcal{F}$.} Then for  $\gamma \in \pi_1(M)$ we obtain: $ { \gamma^*\tilde{\mathcal{F}} = \tilde{\mathcal{F}}(\gamma.z, \overline{\gamma.z}, \lambda) = \chi(\gamma,\lambda)  \tilde{\mathcal{F}}(z, \overline{z}, \lambda)} \hspace{2mm} \mbox{with} \hspace{2mm} {\chi(\gamma,\lambda)  \in G}.$

An analogous formula also holds for the extended frame $F\zl$.
\end{theorem}
\item  We would also like to point out that we will assume throughout this paper that all harmonic maps are ``full". A harmonic map $\mathcal{F}:M\rightarrow G/K$ is called ``full"  if only $g = e$ fixes every element
of $\mathcal{F}(M)$.
That is,  if there exists $g\in G$ such that $g \mathcal{F}(p)=\mathcal{F}(p)$ for all
$p\in M$, then $g=e$.
\end{itemize}


\subsection{The loop group method: Harmonic maps $\Longleftrightarrow$ Normalized Potentials}
\begin{itemize}
    \item {Loop group} with {weighted Wiener topology}:
\[\Lambda G^{\C} = \{ g: S^1 \rightarrow G^{\C}; ||g(\lambda)|| = || \sum_{k \in \mathbb{Z}} g_k \lambda^k||
=  \sum_{k \in \mathbb{Z}} w(k) |g_k| < \infty \}.\]
\item
Twisted Loop group with weighted Wiener topology:
$\Lambda G^{\C}_\sigma = \{ g \in \Lambda G^{\C} ; \sigma ( g(-\lambda)) = g(\lambda) \}$.
\item The real, $+$ and $-$ loop groups: $\Lambda G_\sigma = \{ g \in \Lambda G^{\C}_\sigma ; g \in G \}$.
$\Lambda^+ G^{\C}_{\sigma}  = \{ g \in \Lambda G^{\C}_{\sigma} ;  g_k = 0 \hspace{2mm} \mbox{for} \hspace{2mm} k < 0 \}$.
$\Lambda^- G^{\C}_\sigma  = \{ g \in \Lambda G^{\C}_\sigma ;  g_k = 0 \hspace{2mm} \mbox{for}
\hspace{2mm} k > 0 \}$.
\item

Loops which have a finite Fourier expansion are called {\it algebraic loops} and
 denoted by the subscript $``alg"$, like
$\Lambda_{alg} G_{\sigma},\ \Lambda_{alg} G^{\mathbb{C}}_{\sigma},\
\Omega_{alg} G_{\sigma} $ as in  \cite{BuGu}, \cite{Gu2002}. And we define
  \begin{equation}\label{eq-alg-loop}\Omega^k_{alg} G_{\sigma}:
  =\left\{\gamma\in
\Omega_{alg} G_{\sigma}|
Ad(\gamma)=\sum_{|j|\leq k}\lambda^jT_j \right\}\subset \Omega_{alg} G_{\sigma} .\end{equation}
\item We have  $F(z, \bar{z}, \lambda) \in \Lambda G_\sigma $ \hspace{2mm} \mbox{and} \hspace{2mm}
 $\alpha_\lambda \in \Lambda \mathfrak{g}_{\sigma}.$
\item From harmonic map $\mathcal{F}$ to the normalized potential:
\begin{itemize}
    \item  $\mathcal{F}: \tilde{M} \rightarrow G/{\hat{K}}$  { harmonic map} $\Rightarrow$
    \item
The extended frame $F\zl$ of $\mathcal{F}$ $\Rightarrow$ (By Birkhoff decomposition \cite{DPW,PS})

    \item Decompose {$F\zl = F_-(z,\lambda) V_+(z,\bar{z},\lambda)$}, with {$F_-(z,\lambda) \in \Lambda^- G^{\C}_\sigma $} meromorphic in $z \in \tilde{M}$  and  holomorphic in $\lambda \in \C^*$, and  with $V_+$ holomorphic for  $\lambda  \in \C$ and real analytic in $z, \bar{z} \in \tilde{M}$, with singularities. $F_\lambda : \tilde{M} \rightarrow  \Lambda G_\sigma$ is called  { the normalized extended frame for $\mathcal{F}$} $\Rightarrow$
    \item We have \emph{the normalized  potential} $\eta$  for  $\mathcal{F}$ with base point $z_0$, by setting  $\eta:= F_-^{-1} \dd F_- = \lambda^{-1} \eta_{-1} \dd z$, where $\eta_{-1} \in \mathfrak{p}^{\C}.$
\end{itemize}

\item From the normalized  potential $\eta$ to the harmonic maps $\mathcal{F}_{\lambda}$:
\begin{itemize}
    \item
We start from a  normalized potential type meromorphic one-form,
{$\eta= \lambda^{-1} \eta_{-1}\dd z,$ } with  $ \eta_{-1} \in \mathfrak{p}^{\C}$ and meromorphic in $z$.
    $\Rightarrow$
    \item
Let $F_-(z, \lambda)$  solve the ODE
$\dd F_-= F_- \eta$ satisfying  $F_-(z_0, \lambda) = e$.
    $\Rightarrow$
    \item
    Let's decompose $F_- = F\zl \cdot (V_+)^{-1}$. This is an ``Iwasawa splitting''( an
infinite dimensional ''Gram-Schmidt'' decomposition) \cite{DPW,PS}. We obtain
 the extended frame $F\zl: \tilde{M} \rightarrow  \Lambda G_\sigma$ of a harmonic map from $\tilde{M}$ into $G/{\hat{K}}$.
    $\Rightarrow$
    \item
$ \mathcal{F} \equiv F_\lambda \mod \hat{K} : \ \tilde{M} \rightarrow G/{\hat{K}}$
Defines the associated family of  {harmonic maps with extended frames $F_\lambda$}.
\end{itemize}

\end{itemize}


\section{{Appendix II: Algebraic and totally symmetric harmonic maps}}
In this section we collect some basic definitions and results concerning Algebraic and totally symmetric harmonic maps. For more details we refer to \cite{DoWa-AT}.
\subsection{{Symmetries of harmonic maps}}
We start by recalling the definition of a symmetry for a harmonic map.
For more details we refer to \cite{DoWa-AT}.
\begin{itemize}
     \item
 {Let  $\mathcal{F}: M \rightarrow G/K$ be a harmonic map. Then a pair, $(\gamma,R)$, is called a symmetry of $\mathcal{F}$, if
$\gamma \in \pi_1(M)$ and $R \in G$  is an automorphism of $G/K$ satisfying
$\mathcal{F}(\gamma.p) = R.\mathcal{F}(p)$
  for all $p \in M$.}
  \item
  Let  $\mathcal{F}: M \rightarrow G/K$ be a harmonic map {and}  $(\gamma,R)$ a symmetry
  of $\mathcal{F}.$
Let $\tilde{M}$ denote the universal cover of $M$
and $\tilde{\mathcal{F}}: \tilde{M} \rightarrow G/K, \tilde{\mathcal{F}} = \mathcal{F} \circ \pi$, its natural lift.
\begin{lemma}We have
     \begin{enumerate}
 \item  $\tilde{\mathcal{F}}$ satisfies $\tilde{\mathcal{F}} (\gamma.z) = R.\tilde{\mathcal{F}}(z).$
\item
 For any frame $F : \tilde{M} \rightarrow G$  of $\mathcal{F}$ one  obtains
$\gamma^*F(z,\bar z) = RF(z,\bar z)k(z,\bar z),$
where $k(z,\bar z)$ is a function from $\tilde{M}$ into $K$.
\item  For the extended frame  {$F(z,\bar z,\lambda) : \tilde{M} \rightarrow \Lambda G_\sigma$
of $\mathcal{F}$} there exists some map
 $\rho_\gamma: \C^* \rightarrow \Lambda G_\sigma$ such that
$
\gamma^*F(z,\bar z,\lambda)  = \rho_\gamma (\lambda) F(z,\bar z,\lambda)  {k(z, \bar z)},$
where $k$ is the $\lambda \textendash$independent function from $\tilde{M}$ into $K$ occurring in (2). Moreover, $\rho_{\gamma}(\lambda)|_{ \lambda= 1} = R$ holds.
  \end{enumerate}
\end{lemma}
\item
 {Note, since $\mathcal{F}$ is full, for each symmetry $(\gamma,R)$ the automorphism $R$ of $G/K$
is uniquely determined by $\gamma$. We therefore write $\rho_\gamma (\lambda) = \rho(\gamma,\lambda)$ and ignore $R$ in this notation. Also note that $\rho_\gamma$ actually is defined and holomorphic for all $\lambda \in \C^*$.}
\end{itemize}


\subsection{{Algebraic harmonic maps}}

We start by giving the definition and will list a few properties. For more details we refer to \cite{DoWa-AT}.
\begin{definition} (Algebraic Harmonic Maps ) Let $M$ be a Riemann surface and $\tilde{M}$ its simply-connected covering Riemann surface. A harmonic map
$\mathcal{F}: M \rightarrow G/K$  {is said} to be \emph{algebraic,} if some  extended frame $F$ of $\mathcal{F},$ i.e. some frame satisfying $F(z_0,\bar z_0, \lambda)=e$ for some (fixed) base point $z_0\in \tilde{M},$ is a Laurent polynomial in $\lambda$.
\end{definition}
 Note that if a harmonic map is algebraic, then any extended frame is a Laurent polynomial in $\lambda$. So the initial condition $F(z_0,\bar z_0, \lambda)=e$ is not a restriction.
\begin{proposition}\label{prop-fut} \cite{DoWa-AT}  Let $\mathcal{F}:\tilde{M} \rightarrow G/K$ be a harmonic map defined on a  {contractible} Riemann surface $\tilde{M}$. Let $z_0\in \tilde {M}$ be a base point. Then the following statements are equivalent:
\begin{enumerate}
\item   $\mathcal{F}$ is  an algebraic harmonic map.
\item There exists a  frame  {$F(z,\bar z,\lambda)$} of $\mathcal{F}$ which is a Laurent polynomial in $\lambda$.
\item  The normalized extended frame  {$\hat{F}_-(z,\lambda)$ of any extended frame
$\hat{F}(z,\bar z,\lambda)$ of $\mathcal{F}$} is a Laurent polynomial in $\lambda$.
\item  Every holomorphic  {extended frame $\hat{C}(z,\lambda)$}  associated to any extended frame
 {$\hat{F}(z,\bar z,\lambda)$} of $\mathcal{F},$ thus satisfying $\hat{C}(z_0, \lambda) = e$ for all $\lambda,$
 only  contains finitely many negative powers of $\lambda$.
\item There exists a holomorphic extended  { frame $C^\sharp(z,\lambda)$ } which  { only contains} finitely many negative powers of $\lambda$.
\end{enumerate}
\end{proposition}
\begin{remark}
    \begin{enumerate}
        \item   For the case $\tilde{M} = S^2,$  the above   equivalences remain true, if one replaces in the last two statements the word ``holomorphic" by ``meromorphic" and  also admits for all
    frames  at most two singularities.

\item Note, our assumptions imply that the monodromy representation $\rho(\gamma, \lambda)$ is trivial for $\lambda =1$.
For general $\lambda$ the monodromy representation will not be trivial.
    \end{enumerate}
\end{remark}

\subsection{Totally Symmetric Harmonic maps} \label{trivmonorep}

In the last subsection we considered exclusively simply-connected Riemann surfaces $M$.
Obviously then, the monodromy representation of any harmonic map $\mathcal{F}:M \rightarrow G/K$ is trivial. But there also exist non-simply-connected Riemann surfaces which admit harmonic maps
$\mathcal{F}:M \rightarrow G/K$ with  trivial monodromy representation.
It is this type of harmonic maps which we will discuss in this subsection.
\begin{definition}\label{def-uni} (Totally Symmetric Harmonic Maps) Let $M$ be a Riemann surface  and let $G/K$  be an inner symmetric space.
A harmonic map $\mathcal{F}:M\rightarrow G/K$ is said to be \emph{totally symmetric,}  if there exists some frame $F$ of $\mathcal{F}$, defined on the universal cover $\tilde{M}$ of $M$ which has trivial monodromy  for all $\gamma \in \pi_1(M)$, i.e.
\[F(g.z,\overline{g.z},\lambda)=  F(z,\bar{z},\lambda) k(g,z,\bar{z}), \hbox{ for any $g \in \pi_1(M)$ and any $\lambda \in S^1$.}\]
This is equivalent to that
 for all $\lambda \in S^1,$ any extended frame $F(z,\bar{z},\lambda): \tilde{M} \rightarrow (\Lambda  {G _{\sigma})} $ descends to a well defined map on $M$, i.e. $F(z,\bar{z},\lambda): M \rightarrow (\Lambda G_{\sigma})/K,$
 up to two singularities in the case of $M = S^2$.
We would also like to point out that we will use, by abuse of notation,  the same notation for the frame $F$ with values in
$\Lambda G _{\sigma}$ and its projection with values in $ \Lambda G_{\sigma}/K$.
 \end{definition}

 \begin{remark}
  \begin{enumerate}

  \item
  We will say ``a harmonic map is of \emph{finite uniton type}" if and only if the harmonic map is algebraic (property $(U2)$) and totally symmetric (property $(U1)$). By Proposition \ref{typeequivnumber}, this is equivalent to say
``the harmonic map is of \emph{finite uniton number}" in the sense of  \cite{Uh,BuGu}. For more details see \cite{DoWa-AT}.

  \item We sometimes express equivalently ``totally symmetric" by ``with trivial monodromy (representation)".  Condition of being ``totally symmetric" is a very strong condition. Of course, in the case of a non-compact simply connected Riemann surface $M$   {and  {of} $S^2$ respectively,} it is always satisfied.
\end{enumerate}
\end{remark}


\section{{Appendix III: The Duality Theorem}}

Here we recall how one can relate a harmonic map into a non-compact inner symmetric  space $G/K$ to a  harmonic map into  the compact dual  inner symmetric space of $G/K$. For more information we refer to \cite{DoWa13}.

In \cite{BuGu}, Burstall and Guest have given  for compact inner symmetric spaces $G/K$  an explicit description of those normalized potentials which produce finite uniton type harmonic maps.
To make their work applicable to the non-compact case, we established in \cite{DoWa13} a duality theorem between harmonic maps into non-compact inner symmetric spaces and their compact dual. The most important feature of this result is that the corresponding harmonic maps share the same normalized potential (See the next theorem).
 {We will show  briefly how this duality relation works and what it implies for the properties $(U1)$
and $(U2)$}.

 {Let $G/K$ be a  non-compact inner symmetric space, defined by
 $\sigma$,  and $\tilde{G}$ the (connected) simply-connected cover  of $G$.
Then $G/K = \tilde{G}/\tilde{K}$ for some closed subgroup $\tilde{K}$ of  $\tilde{G}.$ }
 {Let $\tilde{U},$  be a (connected, simply-connected  and semi-simple)  maximal compact Lie subgroup
of $\tilde{G}^{\C}$, the complexification of $\tilde{G}.$}  We  {can also  assume w.l.g. that   $\tilde{G}^{\C}$  is a complex matrix Lie group with Lie subgroup $G$.
  Moreover, we can assume w.l.g. that     $\tilde{U}$   is invariant under the $\C-$linear extension of $\sigma$ to $G^\C$. Finally, by abuse of notation, for $G = SL(2,\R)$ we also denote by $\tilde{G}$ and $\tilde{K}$ the natural image of  $\tilde{G}$ and $\tilde{K}$ in $G^\C$. See, e.g. \cite{Hoch}.}

  We consider the compact dual $\tilde{U}/(\tilde{U}\cap \tilde{K}^{\C})$  of $G/K$ which
 clearly also is defined by $\sigma$ (see \cite{DoWa13} for more details).
Moreover,  $ (\tilde{U}\cap \tilde{K}^{\C})^\C = \tilde{K}^{\C}$ holds
(see {Theorem 1.1 of \cite{DoWa13}}). From this we infer:
\begin{equation}\label{eq-loop-com-n-com}
    \Lambda \tilde{G}^{\C}_{\sigma}=\Lambda \tilde{U}^{\C}_{ {\sigma}},\    \Lambda_*^- \tilde{G}^{\C}_{\sigma}=\Lambda_*^- \tilde{U}^{\C}_{ {\sigma}},\    \Lambda^+\tilde{ G}^{\C}_{\sigma}=\Lambda ^+ \tilde{U}^{\C}_{ {\sigma}}.
\end{equation}
As a consequence, for any extended framing $F(z,\bar{z},\lambda)$ of  $\mathcal{F}: M \to G/K$, the decomposition
 $F(z,\bar{z},\lambda)=F_{\tilde{G,-}} (z,\lambda) F_{\tilde{G},+}(z,\bar{z},\lambda)=F_{\tilde{U},-}(z,\lambda)   F_{\tilde{U},+}(z,\bar{z},\lambda)$ shows
 \begin{equation}\label{eq-norm-framing}
\left\{\begin{split}F_{\tilde{G},-} (z,\lambda)&=F_{\tilde{U},-}(z,\lambda),\\
 \eta&=\lambda^{-1}\eta_{-1}\mathrm{d}z=
(F_{\tilde{G},-})(z,\lambda)^{-1}\mathrm{d}F_{\tilde{G},-}(z,\lambda)=(F_{\tilde{U},-})^{-1}(z,\lambda)\mathrm{d}F_{\tilde{U},-}(z,\lambda).
 \end{split}\right.
 \end{equation}
\begin{theorem} \label{thm-noncompact}$($\cite{DoWa13}$)$
Let
$G/K = \tilde{G} / \tilde{K} $ be a non-compact inner symmetric space with $ \tilde{G} $ simply-connected.
Let $\tilde{U} / (\tilde{U} \cap \tilde{K}^{\C})$ denote the dual compact symmetric space. Then the space
$ \tilde{U} / ( \tilde{U} \cap \tilde{K}^{\C} )$ is inner and  $\Lambda \tilde{G}_{\sigma}^{\C} =
\Lambda \tilde{U}_{\sigma}^{\C} $ holds.
Let $\mathcal{F}:\tilde{M} \rightarrow G/K = \tilde{G} / \tilde{K} $ be a harmonic map, where $ \tilde{M} $
is a simply-connected Riemann surface. Let  {$F(z,\bar{z},\lambda)$ denote an extended frame
of $ \mathcal{F} $. For $F(z,\bar{z},\lambda)$ define
$ F_ {\tilde{U}} $   via the Iwasawa decomposition $F(z,\bar{z},\lambda)=F_{\tilde{U}}(z,\bar{z},\lambda) S_+(z,\bar{z},\lambda)$, with
$F_{\tilde{U}}(z,\bar{z},\lambda)\in \Lambda \tilde{U}_{\sigma}$, $S_+ (z,\bar{z},\lambda)\in \Lambda^+ \tilde{U}_{\sigma}^{\C}$.
Then  \[\hbox{$\mathcal{F}_{\tilde{U}} :\tilde{M} \rightarrow \tilde{U} / \tilde{U} \cap \tilde{K}^{\C}$, \hbox{ with }
$\mathcal{F}_{\tilde{U}} \equiv F_{\tilde{U}}(z,\bar{z},\lambda) \mod \tilde{U} \cap \tilde{K}^{\C}$}\]} is a harmonic map  {for each fixed $\lambda\in S^1$}. Moreover, the harmonic maps $\mathcal{F}$ and $\mathcal{F}_{\tilde{U}}$ have the same normalized potential.
\end{theorem}

As a consequence of Theorem \ref{thm-noncompact} we obtain (see also Theorem 3.6  and Theorem 3.14 of \cite{DoWa13})
\begin{corollary}\label{cor-finite}
Let  {$F(z,\bar{z},\lambda)$ and $F_{\tilde{U}}(z,\bar{z},\lambda)$} be the extended frames defined as above. Then
 \begin{enumerate}
\item
 $F(z,\bar{z},\lambda)$ satisfies $(U1)$ if and only if $F_{\tilde{U}}(z,\bar{z},\lambda)$ satisfies $(U1)$;
\item
  $F(z,\bar{z},\lambda)$ satisfies $(U2)$ if and only if $F_{\tilde{U}}(z,\bar{z},\lambda)$ satisfies $(U2)$.
\end{enumerate}
Therefore $\mathcal{F}$ is of finite uniton number if and only if $\mathcal{F}_{\tilde{U}}$ is of finite uniton number.
\end{corollary}

In particular, finding the normalized potentials  of the finite uniton number harmonic maps $\mathcal{F}$ into the
non-compact inner symmetric space $G/K$ means finding the normalized potentials for the corresponding
finite uniton number harmonic maps $\mathcal{F}_{\tilde{U}}$ into the compact dual
${\tilde{U}}/({\tilde{U}}\cap {\tilde{K}}^{\C})$, which can done in terms of the work of \cite{BuGu}, as shown in the above sections.

\ \\

{\bf{Acknowledgements}}
  {PW} was partly supported by the Project 12371052 of NSFC.

{\footnotesize
\def\refname{References}

}
{\small\

 Josef F. Dorfmeister

Fakult\" at f\" ur Mathematik, TU-M\" unchen,

Boltzmann str.3, D-85747, Garching, Germany

{\em E-mail address}: dorfmeis@gmail.de\\

Peng Wang

School of Mathematics and Statistics, FJKLMAA,

Key Laboratory of Analytical Mathematics and Applications (Ministry of Education),

Fujian Normal University, Qishan Campus,

Fuzhou 350117, P. R. China

{\em E-mail address}: {pengwang@fjnu.edu.cn}

\end{document}